\documentclass[12pt,reqno]{amsart}
\usepackage[T1]{fontenc}
\usepackage[utf8]{inputenc}
\usepackage[italian,english]{babel}
\usepackage{xcolor}
\usepackage{amssymb}
\usepackage{geometry}
\usepackage{comment}
\usepackage{amsmath, amsfonts, amsthm}
\usepackage{graphicx}
\usepackage{pgfplots}
\usepackage{paralist}
\usepackage{mathtools,mathabx,bigints}

\usepackage{comment} 

\numberwithin{equation}{section}
\newtheorem{thm}{Theorem}[section]
\newtheorem{lem}[thm]{Lemma}

\newtheorem{prop}[thm]{Proposition}

\newtheorem{defn}[thm]{Definition}
\theoremstyle{definition}
\newtheorem{rem}[thm]{Remark}
\theoremstyle{remark}

\newcommand{\norm}[1]{\left\Vert#1\right\Vert}

\newcommand{\abs}[1]{\left\vert#1\right\vert}

\newcommand{\R}{\mathbb{R}}
\newcommand{\N}{\mathbb{N}}

\newcommand{\argmin}{\arg\min}

\DeclareMathOperator{\dive}{div}

\DeclareMathOperator{\dist}{dist}

{\left\{\begin{array}{@{}l@{}}}{\end{array}\right.}
\patchcmd{\abstract}{\scshape\abstractname}{\textbf{\abstractname}}{}{}
\makeatletter 
\def\@makefnmark{} 
\makeatother 
 
\begin{document}
\title[]{The $p_0-$Laplace \lq\lq Signature\rq\rq \ for Quasilinear Inverse Problems}
\author[A. Corbo Esposito, L. Faella, V. Mottola, G. Piscitelli, R. Prakash, A. Tamburrino]{
Antonio Corbo Esposito$^1$, Luisa Faella$^1$, Gianpaolo Piscitelli$^2$, Vincenzo Mottola$^1$, Ravi Prakash$^3$, Antonello Tamburrino$^{1,4}$}\footnote{\\$^1$Dipartimento di Ingegneria Elettrica e dell'Informazione \lq\lq M. Scarano\rq\rq, Universit\`a degli Studi di Cassino e del Lazio Meridionale, Via G. Di Biasio n. 43, 03043 Cassino (FR), Italy.\\
$^2$Dipartimento di Matematica e Applicazioni \lq\lq R. Caccioppoli\rq\rq, Universit\`a degli Studi di Napoli Federico II, Via Cinthia n. 26, Complesso Universitario Monte Sant'Angelo, 81026 Napoli, Italy.\\
$^3$Departamento de Matem\'atica, Facultad de Ciencias F\'isicas y Matem\'aticas, Universidad de Concepci\'on, Avenida Esteban Iturra s/n, Bairro Universitario, Casilla 160 C, Concepci\'on, Chile.\\
$^4$Department of Electrical and Computer Engineering, Michigan State University, East Lansing, MI-48824, USA.\\
Email: corbo@unicas.it, l.faella@unicas.it, gianpaolo.piscitelli@unina.it {\it (corresponding author)}, vincenzo.mottola@unicas.it, rprakash@udec.cl, antonello.tamburrino@unicas.it.}
\maketitle

\begin{abstract}
This paper refers to an imaging problem in the presence of nonlinear materials. Specifically, the problem we address falls within the framework of Electrical Resistance Tomography and involves two different materials, one or both of which are nonlinear. Tomography with nonlinear materials in the early stages of developments, although breakthroughs are expected in the not-too-distant future.

The original contribution this work makes is that the nonlinear problem can be approximated by a  {weighted} $p_0-$Laplace problem. From the perspective of tomography, this is a significant result because it highlights the central role played by the $p_0-$Laplacian in inverse problems with nonlinear materials. Moreover, when $p_0=2$, this result allows all the imaging methods and algorithms developed for linear materials to be brought into the arena of problems with nonlinear materials.

The main result of this work is that for \lq\lq small\rq\rq\ Dirichlet data, (i) one material can be replaced by a perfect electric conductor and (ii) the other material can be replaced by a material giving rise to a  {weighted} $p_0-$Laplace problem. 





\end{abstract}

\section{Introduction}
\label{1-PosProb}
This paper is focused on nonlinear imaging problems in Electrical Resistance Tomography where the aim is to retrieve the nonlinear electrical conductivity $\sigma$, starting from boundary measurements in stationary conditions (steady currents). This is a nonlinear variant of the Calder\'on problem \cite{calderon1980inverse,calderon2006inverse}.

Analysis of this class of problems highlights the limiting behaviour of the solution (electric scalar potential) for boundary data approaching zero. In this case, the solution approaches a limit which is the solution of a  {weighted} $p-$Laplace problem. Moreover, the materials with nondominant growth can be replaced by either a \emph{perfect electric conductor} or a \emph{perfect electric insulator}. These results are significant from both a mathematical and an engineering point of view, since they make it possible to approximate a nonlinear phenomenon with a  {weighted} $p-$Laplace problem. In one sense, this suggests the \lq\lq fingerprint\rq\rq \ of a  {weighted} $p-$Laplace problem in a nonlinear problem. The linear case, i.e. $p=2$, is of paramount importance. In this case, we have a powerful bridge to apply all the imaging methods and algorithms developed for linear materials to nonlinear materials. The behaviour for large data has been studied in \cite{corboesposito2023Large} where we use different set of test functions for the Dirichlet energy as we do not have different growth exponents ($p$ and $p_0$) for the asymptotic behaviour.

Hereafter, we consider steady current operations where the constitutive relationship is nonlinear, local, isotropic and memoryless:  
\begin{equation} \label{J}
{\bf J}(x)=\sigma(x,\vert {\bf E}(x)\vert) {\bf E}(x)\quad\forall x\in\Omega.
\end{equation}
In (\ref{J}), $\sigma$ is the nonlinear electrical conductivity, ${\bf J}$ the electric current density, ${\bf E}$ the electric field and $\Omega\subset\R^n$, $n \geq 2,$ is an open bounded domain with Lipschitz boundary. $\Omega$ represents the region occupied by the conducting material. The electric field can be expressed through the electrical scalar potential $u$ as ${\bf E}(x)=-\nabla u(x)$, where $u$ solves the steady current problem:
\begin{equation}\label{gproblem1}
\begin{cases}
\dive\Big(\sigma (x, |\nabla u(x)|) \nabla u (x)\Big) =0\ \text{in }\Omega\vspace{0.2cm}\\
u(x) =f(x)\qquad\qquad\qquad\quad\  \text{on }\partial\Omega,
\end{cases}
\end{equation}
where $f$ is the applied boundary potential. Both $u$ and $f$ belong to proper  {function} spaces that will be defined in the following. 

\label{2-IPnonlinear}
The literature contains very few contributions on imaging in the presence of nonlinear materials. As quoted in \cite{lam2020consistency} (2020), {\it \lq\lq\ ... the mathematical analysis for inverse problems governed by nonlinear Maxwell's equations is still in the early stages of development.\rq\rq}. It can be expected that as new methods and algorithms become available, the demand for nondestructive evaluation and imaging of nonlinear materials will eventually rise significantly .

Among the contributions to the nonlinear Calder\'on problem, special attention has been paid to the case based on the $p-$Laplacian, where $\sigma(x,\vert {\bf E}(x)\vert)=\theta(x)\vert {\bf E}(x)\vert^{p-2}$ in equation $(\ref{J})$, with $\theta$ being an appropriate weight function. The nonlinear $p-$Laplace variant of the Calder\'on problem was initially posed by Salo and Zhong \cite{Salo2012_IP} and subsequently studied in \cite{brander2015enclosure,brander2016calderon,brander2018superconductive,guo2016inverse,brander2018monotonicity,hauer2015p}. As well as the nonlinear $p-$Laplace problem, mention must also be made of the work by Sun \cite{sun2004inverse,Sun_2005} for weak nonlinearities, the work by C\^arstea and Kar \cite{carstea2020recovery} which treated a nonlinear problem (linear plus a nonlinear term) and the work by Corbo Esposito et al. \cite{corboesposito2021monotonicity}. The latter treat a general nonlinearity within the framework of the Monotonicity Principle Method.

\label{3-Applications}
From the application perspective, nonlinear electrical conductivities can be found in semiconducting and ceramic materials (see \cite{bueno2008sno2}), with applications to cable termination in high voltage (HV) and medium voltage (MV) systems \cite{boucher2018interest, lupo1996field}, for instance. Nonlinear electrical conductivities characterize superconductors, key materials for such applications as energy storage, magnetic levitation systems, superconducting magnets (nuclear fusion devices, nuclear magnetic resonance) and high-frequency radio technology \cite{seidel2015applied, krabbes2006high}. Nonlinear electrical conductivity also appears in the area of biological tissues (see \cite{foster1989dielectric}). For instance, \cite{corovic2013modeling} proved that nonlinear models fit the experimental data better than linear models.

\label{4-Generalizations}
Problem (\ref{gproblem1}) is common to steady currents as well as to other physical settings. In the framework of electromagnetism, both nonlinear electrostatic and nonlinear magnetostatic\footnote{$^1$In magnetostatics, it is possible to introduce a magnetic scalar potential for treating simply connected and source free regions.}{$^1$} phenomena can be modelled as in (\ref{gproblem1}). In the first case the constitutive relationship is ${\bf D}(x)=\varepsilon(x,\vert {\bf E}(x)\vert) {\bf E}(x)$ (see \cite{miga2011non} and references therein, and \cite{yarali20203d}), where $\bf D$ is the electric displacement field, $\varepsilon$ is the dielectric permittivity and $\bf E$ the electric field. In the second case ${\bf B}(x)=\mu(x,\vert {\bf H}(x)\vert) {\bf H}(x)$ (see \cite{1993ferr.book.....B}), where $\bf B$ is the magnetic flux density, $\mu$ is the magnetic permeability, and $\bf H$ is the magnetic field.

\label{5-Inverse Problem}
From a general perspective, the inverse problem of retrieving a coefficient of a PDE ( {Partial Differential Equation}) from boundary measurements, such as the electrical conductivity $\sigma$ appearing in \eqref{gproblem1}, is nonlinear and ill-posed in the sense of Hadamard, i.e. it is an inverse problem.

\label{5a-IT}
A classic approach for solving an inverse problem consists in casting it in terms of the minimization of a proper cost function \cite{tikhonov1977solutions,tikhonov1998nonlinear}. The minimizer of this cost function gives the estimate of the unknown quantity. The cost function is usually built as the weighted sum of the discrepancy on the data and proper a priori information that must be provided to complement the loss of information inherent to the physics of the measurement process. There are many iterative approaches devoted to the search for the solution (the minimizer) of an inverse problem. An overview can be found in several specialized textbooks \cite{bertero1998introduction,tarantola2005inverse,vogel2002computational,engl1996regularization}. Other than the Gauss-Newton and its variant (see \cite{qi2000iteratively} for a review), let us mention some relevant iterative approaches applied to inverse problems such as the Quadratic Born approximation \cite{pierri1997local}, Bayesian approaches \cite{premel2002eddy}, the Total Variation regularization \cite{RudinOsher1992nonlinear,pirani2008multi}, the Levenberg-Marquardt method for nonlinear inverse problems \cite{Hanke_1997}, the Level Set method \cite{dorn2000shape,harabetian1998regularization}, the Topological Derivative method \cite{Jackowska-Strumillo2002231,ammari2012stability,fernandez2019noniterative} and the Communication Theory approach \cite{tamburrino2000communications}.

\label{5b-nonIT}
Iterative methods suffer from two major drawbacks: (i) they may be trapped into local minima and (ii) the computational cost may be very high. Indeed, the objective function to be minimized in order to achieve the reconstruction might present several/many local minima which may constitute points where an iterative algorithm may be trapped. Moreover, the computational cost at each iteration may be very high because it entails computing the objective function and, optionally, its gradient. Both computations are expensive in terms of computational resources. 

An excellent alternative to iterative methods is provided by noniterative ones. Noniterative methods are attractive because they call for the computation of a proper function of the space (the so-called indicator function) giving the shape of the interface between two different materials, i.e. the support of the region occupied by a specific material. They usually require a larger amount of data than iterative approaches, but the computation of the indicator function is much less expensive. In general, noniterative methods are suitable for real-time operations.

Only a handful of noniterative methods are currently available. These include the Linear Sampling Method (LSM) by Colton and Kirsch \cite{Colton_1996}, which evolved into the Factorization Method (FM) proposed by Kirsch \cite{Kirsch_1998}. Ikehata proposed the Enclosure Method (EM) \cite{ikehata1999draw,Ikehata_2000} and Devaney applied MUSIC (MUltiple SIgnal
Classification), a well-known algorithm in signal processing, as an imaging method \cite{Devaney2000}. Finally, Tamburrino and Rubinacci proposed the Monotonicity Principle Method (MPM) \cite{Tamburrino_2002}.

\label{6-PEC/PEI}
The prototype problem which motivated this study consists in imaging a two-phase material where the outer phase is linear and the inner phase is nonlinear (see Figure \ref{fig_1_intro}). A configuration of this type may be encountered when testing/imaging superconducting cables (see, for instance, \cite{lee2005nde,amoros2012effective,takahashi2014non,Higashikawa2021_SC}). 
The main result of this work is the proof that for \lq\lq small\rq\rq\ Dirichlet data $f$, the nonlinear material can be replaced by either a perfect electric conductor (PEC) or a perfect electric insulator (PEI). Consequently, when one material is linear, the limiting version of the original nonlinear problem is linear. These results provide a powerful bridge to bring all the imaging methods and algorithms developed for linear materials into the arena of problems presenting nonlinear materials. 

\begin{figure}[ht]
\centering
\includegraphics[width=\textwidth]{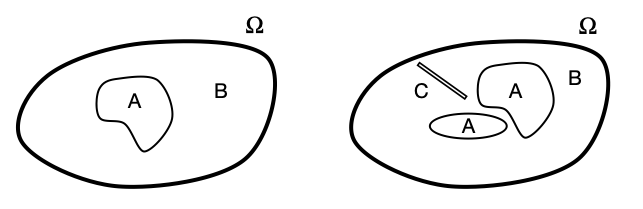}
\caption{Description of two possible applications. Left: inverse obstacle problem where the interface ($\partial A $) between two phases is unknown. $A$ and $B$ are the regions occupied by the inner material and the outer material, respectively. Right: nondestructive testing where  {regions A and B are known, while the position and shape of region $C$ (a crack) is unknown. The materials in regions $A$ and $B$ are also known.}}
    \label{fig_1_intro}
\end{figure}

Moreover, in order to reach a thorough understanding of the underlying mathematics, the results have been proved in a more general setting where both materials are nonlinear. In this case, one material is replaced by either a perfect electric conductor or a perfect electric insulator, and the other is replaced by a material yielding a  {weighted} $p_0-$Laplace problem. 

\label{7-math}
A specific feature of this work concerns the required assumptions, which are general and sharp, as discussed in Sections \ref{fram_sec} and \ref{counter_sec}. The assumptions are general: other than the standard conditions for existence and uniqueness of the solution of (\ref{gproblem1}), they involve pointwise convergence, only. The assumptions are sharp: the fundamental conditions specifically introduced for replacing one material with either a PEC or PEI cannot be removed, as shown by the counterexamples in Section \ref{counter_sec}. 

\label{8-arch}
The paper is organized as follows: in Section \ref{underlying} we present the ideas underpinning the work; in Section \ref{fram_sec} we set out the notations and the problem, together with the required assumptions; in Section \ref{mean_sec0} we give a fundamental inequality for small Dirichlet data; in Section \ref{small_sec} we discuss the limiting case for small Dirichlet data; in Section \ref{counter_sec} we provide the counterexamples proving that the specific assumptions are sharp; in Section \ref{num_sec} we provide numerical validation of the proposed theory; finally, in Section \ref{Con_sec} we provide some conclusions.

\section{Underlying ideas and expected results}
\label{underlying}
In this section we present the main ideas underpinning this work. The key is the \lq\lq educated guess\rq\rq \ that when the boundary data is \lq\lq small\rq\rq, the electric field $\mathbf{E}=-\nabla{u}$ is small a.e. in $\Omega$ and, therefore, its behaviour has to be governed by the asymptotic behaviour of $\sigma \left(x,E\right)$ in the constitutive relationship \eqref{J}. Specifically, let $A\subset\subset\Omega$ and  $B:=\Omega\setminus\overline A$, we assume that there exist two constants $p_0$ and $q_0$, and two functions $\beta_0$ and $\alpha_0$ which capture the behaviour of $\sigma$, as $E \to 0^+$, in $B$ and $A$, respectively:
\begin{align*}
\sigma_B(x,E) \sim \beta_0 (x)E^{p_0-2} \quad
\text{for a.e.}\  x \in B,\\
\sigma_A(x,E) \sim \alpha_0 (x)E^{q_0-2} \quad \text{for a.e.}\  x \in A,
\end{align*}
where $\sigma_B$ and $\sigma_A$ are the restriction of $\sigma$ to $B$ and $A$, respectively and $E = \left| \mathbf{E} \right| = \left| \nabla u \right|$.

Analysis of nonlinear problems is fascinating because of the wide variety of different cases. The most representative cases are shown in Figure \ref{fig_2_sigma}.
\begin{figure}[ht]
    \centering
    \includegraphics[width=0.325\textwidth]{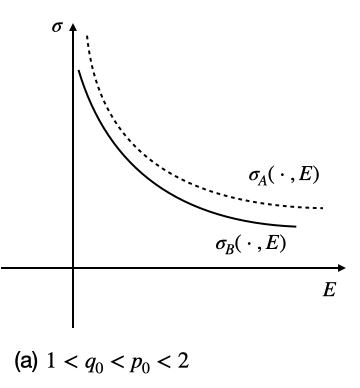}
    \includegraphics[width=0.325\textwidth]{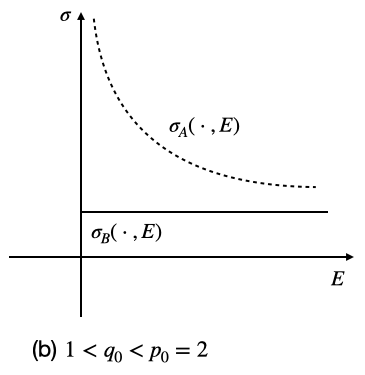}
    \includegraphics[width=0.325\textwidth]{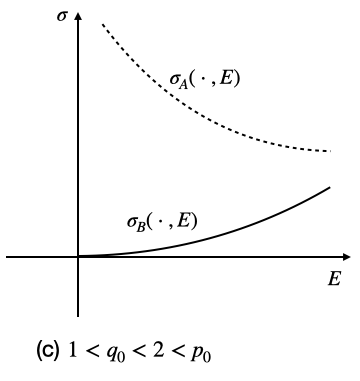}    
    \includegraphics[width=0.325\textwidth]{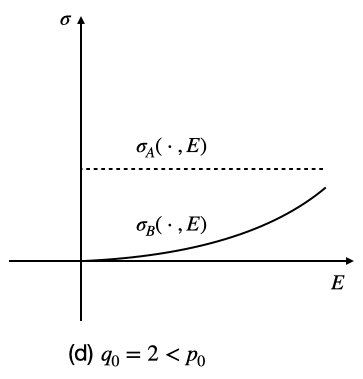}
    \includegraphics[width=0.325\textwidth]{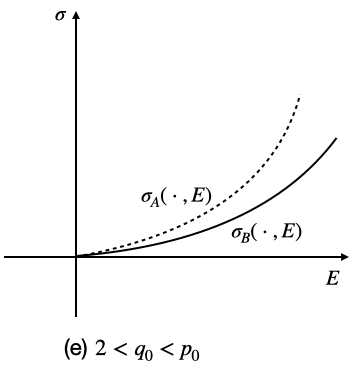}
    \caption{The electrical conductivities of the outer material and of the inner conducting material, when $\sigma_B(\cdot,E)=E^{p_0-2}$ is represented by the continuous lines and $\sigma_A(\cdot,E)=
    E^{q_0-2}$ is represented by the dashed lines. The configurations when the order relation between $p_0$ and $q_0$ is reversed easily follow.}
    \label{fig_2_sigma}
\end{figure}

When, for instance, $q_0<p_0$ it can be reasonably expected that either (i) region $A$ is a perfect electric conductor or (ii) region $B$ is a perfect electric insulator, because $\sigma_B$ would be dominant if compared to $\sigma_A$, at small electric fields.
When $A\subset\subset\Omega$, the ambiguity between (i) and (ii) is resolved in Section \ref{small_sec}, where we prove that region $B$ cannot be assimilated to a PEI and, therefore, $A$ has to be assimilated to a PEC.  {Finally, the case $p_0=q_0$ (that is the case when $A=\emptyset$) has been treated in
\cite{corboesposito2021monotonicity,MPMETHODS}}.

Moreover, the limiting problem where the conductor in region $A$ is replaced by a PEC, can reliably be modelled by a $p_0-$Laplace problem in region $B$, with a boundary condition given by a constant scalar potential $u$, on each connected component of $\partial A$. In other words, $u \sim u_{p_0}$ in $B$, where $u_{p_0}$ is the solution of the  {weighted} $p_0-$Laplace problem arising from the electrical conductivity $\beta_0(x) E^{p_0-2}$ in $B$ and  {$|\nabla u_{p_0}|=0$ on $A$.}

The latter observation is also inspiring as it properly defines the concept of \lq\lq small\rq\rq \ boundary data and the limiting problem. Specifically, it is well known that the operator mapping the boundary data $f$ into the solution of a  {weighted} ${p_0}-$Laplace problem is a homogeneous operator of degree 1, i.e. the solution corresponding to $\lambda f(x)$ is equal to $\lambda u_{p_0}(x)$, where $u_{p_0}$ is the solution corresponding to the boundary data $f$. Thus,  the term \lq\lq problem for small boundary data\rq\rq\ means \eqref{gproblem1} where the boundary data is $\lambda f$ and $\lambda \to 0$. 
Moreover, this suggests the need to study convergent properties of the normalized solution $v^\lambda$, defined as the ratio $u^\lambda / \lambda$, where $u^\lambda$ is the solution of \eqref{gproblem1} corresponding to the Dirichlet data $\lambda f(x)$. Indeed, if $u^\lambda$ can be approximated by the solution of the  {weighted} $p_0-$Laplace problem, then the normalized solution $v^\lambda(x)$ converges in $B$, i.e. it is expected to be constant w.r.t. $\lambda$, as $\lambda$ approaches $0$. We term this limit as $v^0$ and we expect it to be equal to $u_{p_0}$, i.e. the solution of the  {weighted} $p_0-$Laplace problem with boundary data $f$.

From the formal point of view, when $q_0<p_0$, $v^\lambda$  {weakly} converges to $ {w^0\in W^{1,p_0}(\Omega)}$ for $\lambda\to 0^+$, where $w^0$ is constant in each connected component of $A$, and it is the solution of:
\begin{equation}\label{pproblem_Bgrad0}
\begin{cases}
\dive\Big(\beta_0 (x) |\nabla w^0(x)|^{p_0-2}\nabla w^0 (x)\Big) =0 & \text{in }B,\vspace{0.2cm}\\
 {|\nabla w^0(x)| =0} &\text{a .e. in } A,\vspace{0.2cm}\\
\int_{\partial A}\sigma(x,|\nabla w^0(x)|)\partial_\nu w^0(x)dS=0\vspace{0.2cm}\\
w^0(x) =f(x) & \text{on }\partial \Omega,
\end{cases}
\end{equation}
in $B$. In this case, from the physical standpoint, region $A$ can be replaced by a Perfect Electric Conductor (PEC). 

The solution of problem \eqref{pproblem_Bgrad0} satisfies the minimum problem \eqref{H}, described in Section \ref{small_sec}.

 {On the other hand}, when $p_0<q_0$, $v^\lambda$ converges, in $B$, to $v^0_B\in W^{1,p_0}(B)$, that is the solution of the  {weighted} $p_0-$Laplace problem in region $B$:
\begin{equation}\label{pproblem_B0}
\begin{cases}
\dive\Big(\beta_0 (x) |\nabla v^0_B(x)|^{p_0-2}\nabla v^0_B (x)\Big) =0 & \text{in }B,\vspace{0.2cm}\\
\beta_0 (x) |\nabla v^0_B(x)|^{p_0-2}\partial_\nu v^0_B(x) =0 & \text{on }\partial A,\vspace{0.2cm}\\
v^0_B(x) =f(x) & \text{on }\partial \Omega.
\end{cases}
\end{equation}

From the physical standpoint, problem \eqref{pproblem_B0} corresponds to stationary currents where the electrical conductivity is $\sigma(x,E)=\beta_0(x)E^{p_0}$, and region $A$ is replaced by a perfectly electrical insulating material (PEI). 

 {The solution of problem \eqref{pproblem_B0} 
satisfies the minimum problem \eqref{Hii}
, described in Section \ref{small_sec}}.


\section{Framework of the Problem}
\label{fram_sec}
\subsection{Notations}

Throughout this paper, $\Omega$ denotes the region occupied by the conducting materials. We assume that $\Omega\subset\R^n$, $n\geq 2$, is a bounded domain (i.e. an open and connected set) with Lipschitz boundary and $A\subset\subset\Omega$ is an open bounded set with Lipschitz boundary and a finite number of connected components, such that $B:=\Omega\setminus\overline A$ is still a domain.
Hereafter we consider the growth exponents $p, q, p_0$ and $q_0$ such that  {$1< p,q<\infty$, $p \neq q$}, $1< p_0 \leq p<\infty$, $1< q_0 \leq q<\infty$ and $p_0\neq q_0$. $p$ ($p_0$) is related to the growth of the electrical conductivity in region $B$ for large (small) electric fields (see Section \ref{subsec_hyp} for further details). Similarly, $q$ ($q_0$) is related to the growth of the electrical conductivity in region $A$ for large (small) electric fields (see Figure \ref{fig_4_omega}). 

\begin{figure}[ht]
    \centering
\includegraphics[width=\textwidth]{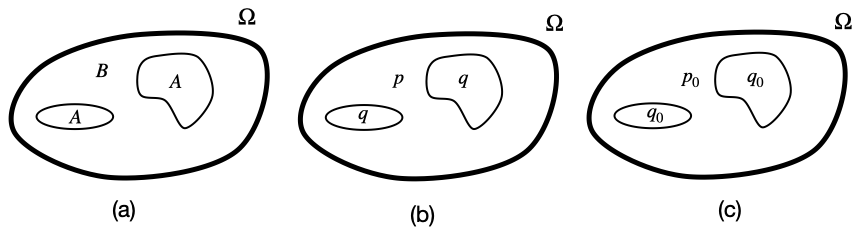}
    \caption{A two phase problem (left) together with the electrical conductivity growth exponents for the electric field in a neighborhood of $+\infty$ (center) and in a neighborhood of $0$ (right).}
    \label{fig_4_omega}
\end{figure}

We denote 
by $dx$ and $dS$ the $n-$dimensional and the $(n-1)-$dimensional Hausdorff measure, respectively. Moreover, we set
\[
L^\infty_+(\Omega):=\{\theta\in L^\infty(\Omega)\ |\ \theta\geq c_0\ \text{a.e. in}\ \Omega, \ \text{for a positive constant}\ c_0\}.
\]
Furthermore, for any $1<s<+\infty$ we denote by $W^{1,s}_0(\Omega)$ the closure set of $C_0^1(\Omega)$ with respect to the $W^{1,s}-$norm.

The applied boundary voltage $f$ belongs to the abstract trace space $B^{1-\frac 1p,p}(\partial\Omega)$, which, for any bounded Lipschitz open set, is a Besov space (refer to \cite{JERISON1995161,leoni17}),  equipped with the following norm:
\[
||u||_{B^{1-\frac 1p,p}(\partial\Omega)}=||u||_{L^p(\partial\Omega)}+|u|_{B^{1-\frac 1p,p}(\partial\Omega)}<+\infty,
\]
where $|u|_{B^{1-\frac 1p,p}(\partial\Omega)}$ is the Slobodeckij seminorm:
\[
|u|_{B^{1-\frac 1p,p}(\partial\Omega)}=\left(\int_{\partial\Omega}\int_{\partial\Omega}\frac{|u(x)-u(y)|^p}{||x-y||^{N-1+(1-\frac 1p)p} }dS (y)d S (x) \right)^\frac 1p,
\]
see Definition 18.32, Definition 18.36 and Exercise 18.37 in \cite{leoni17}.

This guarantees the existence of a function in $W^{1,p}(\Omega)$ whose trace is $f$ \cite[Th. 18.40]{leoni17}.

For the sake of brevity, we denote this space by $X^p(\partial \Omega)$ and its elements can be identified as the functions in $W^{1,p}(\Omega)$, modulo the equivalence relation $f\in [g]_{X^p(\partial \Omega)}$ if and only if $f-g\in W^{1,p}_0(\Omega)$, see \cite[Th. 18.7]{leoni17}.

Finally, we denote by $X^p_\diamond (\partial \Omega)$ the set of elements in $X^p(\partial \Omega)$ with zero average on $\partial\Omega$ with respect to the measure $dS$.

\subsection{The Scalar Potential and Dirichlet energy}

In terms of the electric scalar potential, that is ${\bf E}(x)=-\nabla u(x)$, the nonlinear Ohm's law \eqref{J} is
 \begin{equation*}
 {\bf J} (x)=- \sigma (x, |\nabla u(x)|)\nabla u(x),
 \end{equation*}
where $\sigma$ is the electrical conductivity, ${\bf E}$ is the electric field, and ${\bf J}$ is the electric current density.

The electric scalar potential $u$ 
solves the steady current problem:
 \begin{equation}\label{gproblem}
\begin{cases}
\dive\Big(\sigma (x, |\nabla u(x)|) \nabla u (x)\Big) =0\ \text{in }\Omega\vspace{0.2cm}\\
u(x) =f(x)\qquad\qquad\qquad\quad\  \text{on }\partial\Omega,
\end{cases}
\end{equation}
where $f\in X_\diamond^p(\partial \Omega)$. Problem \eqref{gproblem} is meant in the weak sense, that is
\begin{equation*}
\int_{\Omega }\sigma \left( x,| \nabla u(x) |\right) \nabla u (x) \cdot\nabla \varphi (x)\ \text{d}x=0\quad\forall\varphi\in C_c^\infty(\Omega).
\end{equation*}

The solution $u$ restricted to $B$ belongs to $W^{1,p}(B)$, whereas $u$ restricted to $A$ belongs to $W^{1,q}(A)$;  {however} the solution $u$ as a whole is an element of the  {largest between the two functional spaces $W^{1,p}(\Omega)$ and $W^{1,q}(\Omega)$. Furthermore, }(i) if $p\leq q$ then $W^{1,p}(\Omega)\cup W^{1,q}(\Omega)=W^{1,p}(\Omega)$, and (ii) if $p\geq q$ then $W^{1,p}(\Omega)\cup W^{1,q}(\Omega)=W^{1,q}(\Omega)$.

The solution $u$ satisfies the boundary condition in the sense that $u-f\in W_0^{1,p}(\Omega)\cup W_0^{1,q}(\Omega)$ and we write $u|_{\partial\Omega}=f$.

Moreover, the solution $u$ is variationally characterized as
\begin{equation}\label{gminimum}
\argmin\left\{ \mathbb{E}_\sigma\left( u\right)\ :\ u\in W^{1,p}(\Omega)\cup W^{1,q}(\Omega), \ u|_{\partial\Omega}=f\right\}.
\end{equation}


In (\ref{gminimum}), the functional $\mathbb{E}_\sigma\left( u\right)$ is the Dirichlet energy
\begin{equation*}
\mathbb{E}_\sigma
\left(  u \right) = \int_{B} Q_B (x,|\nabla u(x)|)\ \text{d}x+ \int_A Q_A (x,|\nabla u(x)|)\ \text{d}x
\end{equation*} 
where $Q_B$ and $Q_A$ are the Dirichlet energy density in $B$ and in $A$, respectively:
\begin{align*}
& Q_{B} \left( x,E\right)  :=\int_{0}^{E} \sigma_B\left( x,\xi \right)\xi  \text{d}\xi\quad \text{for a.e.}\ x\in B\ \text{and}\ \forall E\geq0,\\
& Q_{A}\left( x,E\right)  :=\int_{0}^{E} \sigma_A\left( x,\xi \right)\xi  \text{d}\xi\quad \text{for a.e.}\ x\in A\ \text{and}\ \forall E\geq 0,
\end{align*}
and $\sigma_B$ and $\sigma_A$ are the restiction of the electrical conductivity $\sigma$ in $B$ and $A$, respectively.

\subsection{Requirements on the Dirichlet energy densities}\label{subsec_hyp}
In this Section, we provide the assumptions on the Dirichlet energy densities $Q_B$ and $Q_A$, to guarantee the well-posedness of the problem and to prove the main convergence results of this paper.

For each individual result, we will make use of a minimal set of assumptions, among those listed in the following.

Firstly, we recall the definition of the Carathéodory functions.
\begin{defn}
$Q:\Omega\times[0,+\infty)\to\R$ is a Carathéodory function iff:
\begin{enumerate}
\item $\Omega\ni x\mapsto Q(x,E)$ is measurable for every $E\in [0,+\infty)$,
\item $[0,+\infty)\ni E\mapsto Q(x, E)$ is continuous for almost every $x\in\Omega$.
\end{enumerate}
\end{defn}
The assumptions on $Q_B$ and $Q_A$, required to guarantee the existence and uniqueness of the solution, are as follows.
\begin{itemize}
\item[\textbf{(A1)}]  $Q_B$ and $Q_A$ are Carathéodory functions;
\item[\textbf{(A2)}]  {$[0,+\infty)\ni E\mapsto Q_B(x,E)$ and $[0,+\infty)\ni E\to Q_A(x,E)$ are nonnegative}, $C^1$,  strictly convex, $Q_B(x,0)=0$ for a.e. $x\in B$, and $Q_A(x,0)=0$ for a.e. $x\in A$.
\end{itemize}


The behaviour of $Q_A$ and $Q_B$ for small Dirichlet boundary data, satisfies the following assumptions:
\begin{itemize}
\item[\textbf{(A3)}] There exists two exponents $p_0$ and $q_0$ with $1< p_0 \leq p<\infty$, $1< q_0 \leq q<\infty$ and $p_0\neq q_0$, such that: 
\[
\begin{split}
&(i)\  \underline{Q}\  {\max\left\{  \left(\frac{E}{E_0}\right)^{p_0},\left(\frac{ E}{E_0}\right)^p\right\}}\leq Q_B(x, E)\leq\overline Q \max\left\{  \left(\frac{E}{E_0}\right)^{p_0},\left(\frac{ E}{E_0}\right)^p\right\}\\
& \qquad\qquad\qquad\qquad\qquad\qquad\qquad\qquad\qquad\qquad \text{for a.e.} \  x\in B \ \text{and}\ \forall\  E\ge 0,\\
&(ii)\  \underline{Q} \  {\max\left\{  \left(\frac{E}{E_0}\right)^{q_0},\left(\frac{ E}{E_0}\right)^q\right\}}\leq Q_A(x, E)\leq\overline Q \max\left\{  \left(\frac{E}{E_0}\right)^{q_0},\left(\frac{ E}{E_0}\right)^q\right\}\\
& \qquad\qquad\qquad\qquad\qquad\qquad\qquad\qquad\qquad\qquad \text{for a.e.} \  x\in A \ \text{and}\ \forall\  E\ge 0.
\end{split}
\]
\end{itemize}
 {Assumption (A2) implies that both $Q_B$ and $Q_A$ are increasing functions in $E$; moreover, (A2) and (A3) imply that both $Q_B(x, E)\leq \overline Q$ and $Q_A(x,E)\le\overline Q$, when $0\le E\leq E_0$.}

Finally, assumption  {(A3) is implied by the well-known hypothesis used} in the literature (see e.g. assumptions (H4) in \cite{corboesposito2021monotonicity} and (0.2) in \cite{giaquinta1982regularity}).
\begin{itemize}
\item[{\bf (A4)}] There exists a function $\beta_0\in L^\infty_+(B)$ such that:
\begin{equation*}
\begin{split}
\lim_{E\to 0^+} \frac{Q_B (x,E)}{E^{p_0}}=\beta_0(x)\quad \text{for a.e.}\  x\in B.
\end{split}
\end{equation*}
\end{itemize}

In Section \ref{counter_sec}, we will provide another counterexample to show that assumption (A4) is sharp.


\subsection{Connection among $\sigma$, ${\bf J}$ and $Q$}

This paper is focused on the properties of the Dirichlet energy density $Q$, while, in physics and engineering the electrical conductivity $\sigma$ is of greater interest. From this perspective, assumptions (Ax) are able to include a wide class of electrical conductivities (see Figure \ref{fig_5_assumptions}). In other words, the (Ax)s are not restrictive in practical applications.

\begin{figure}[h]
\centering
\includegraphics[width=0.42\textwidth]{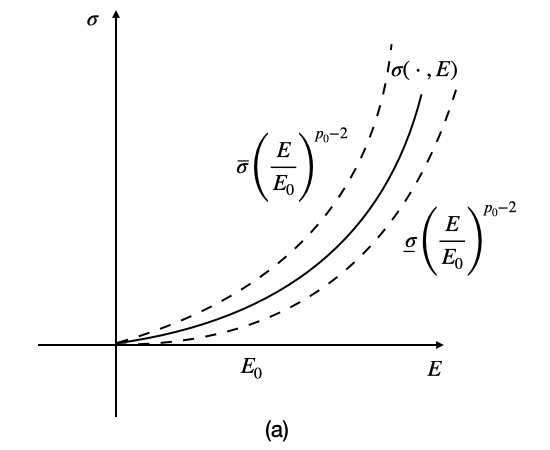}
	\includegraphics[width=0.42\textwidth]{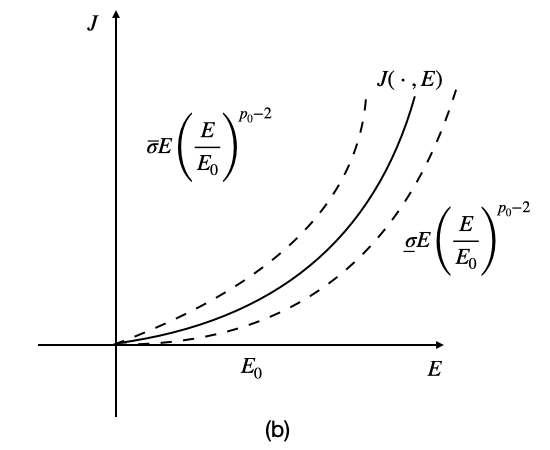}
	\includegraphics[width=0.42\textwidth]{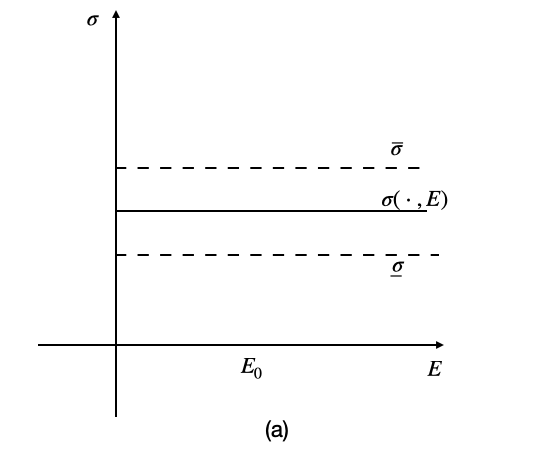}
	\includegraphics[width=0.42\textwidth]{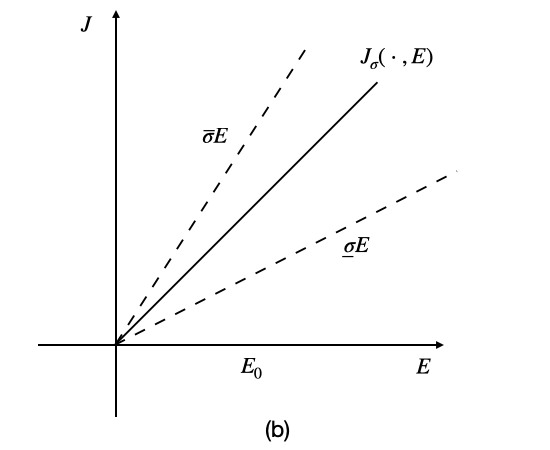}
	\includegraphics[width=0.42\textwidth]{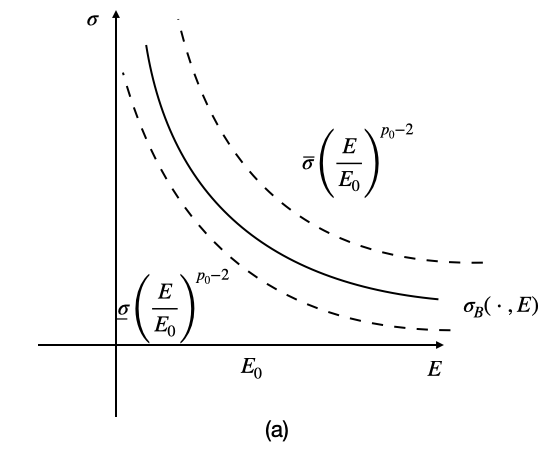}
	\includegraphics[width=0.42\textwidth]{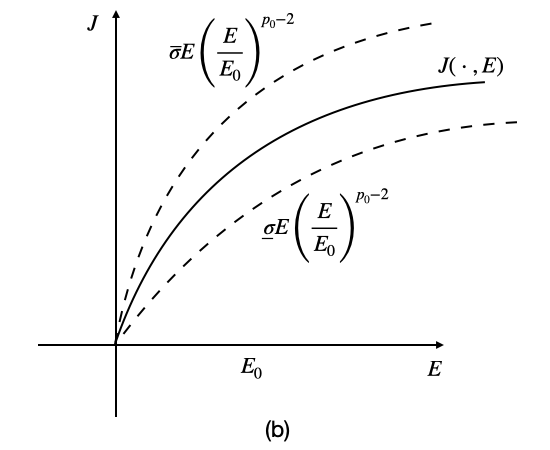}
	\caption{Behaviour of the constitutive relationship in a neighborhood of $E=0$ for $p_0>2$, $p_0=2$ and $p_0<2$: in terms of (a) the electrical conductivity $\sigma$ and (b) the electrical current density $J$.  {Dashed} lines correspond to the upper and lower bounds to either $\sigma$ or $J$.}
	\label{fig_5_assumptions}
\end{figure}

There is a close connection between $\sigma$, $J$ and $Q$. Indeed,
\begin{equation*}
\begin{split}
Q_B
\left( x,E \right)  =\int_{0}^{E}{J_B} ({ x, \xi}) \ \text{d}{ \xi}\quad\text{for a.e.} \ x\in B\ \text{and}\ \forall E> 0,\\
Q_A
\left( x,E \right)  =\int_{0}^{E}{J_A} ({ x, \xi}) \ \text{d}{ \xi}\quad\text{for a.e.} \ x\in A\ \text{and}\ \forall E> 0,
\end{split}
\end{equation*}
where $J_B$ and $J_A$ is the magnitude of the current density in regions $B$ and $A$, respectively:
\begin{equation}\label{connsJQ}
    \begin{split}
J_B (x, E)&=\partial_E Q_B(x,E)=\sigma_B(x, E)E\quad \text{for a.e.} \ x \in B\ \text{and}\ \forall E>0,\\
J_A (x, E)&=\partial_E Q_A(x,E)=\sigma_A(x, E)E\quad \text{for a.e.} \ x \in A\ \text{and}\ \forall E>0.
    \end{split}
\end{equation}

The electrical conductivity $\sigma(x,E)$ is the secant to the graph of the function $J_\sigma(x,E(x))$ and $Q_\sigma (x, E(x))$ is the area of the sub-graph of $J_\sigma(x, E(x))$. For a geometric interpretation of the connections between $\sigma$, $J_\sigma$ and $Q_\sigma$, see Figure \ref{fig_6_JE}.

\begin{figure}[ht]
	\centering
	\includegraphics[width=0.45\textwidth]{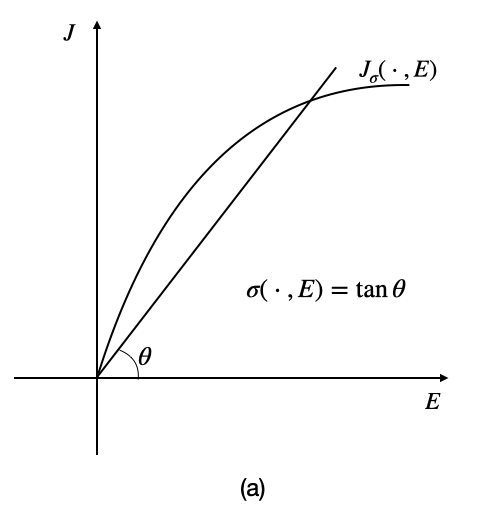}
	\includegraphics[width=0.45\textwidth]{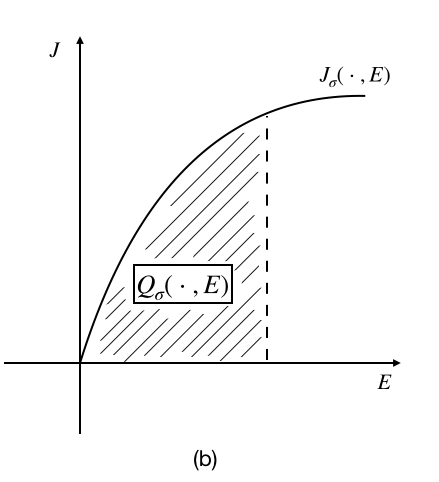}
	\caption{For any given spatial point in the region $\Omega$, (a) the electrical conductivity $\sigma(\cdot,E)$ is the secant line to the graph of the function $J_\sigma(\cdot,E)$;
	(b) $Q_\sigma (\cdot, E)$ is the area of the sub-graph of $J_\sigma(\cdot, E)$.}
	\label{fig_6_JE}
\end{figure}

\subsection{Existence and uniqueness of the solutions}
The proof of the existence and uniqueness of the solution for \eqref{gproblem} in its variational form, relies on standard methods of the Calculus of Variations, when the Dirichlet energy density presents the same growth in any point of the domain $\Omega$. The case treated in this work is nonstandard, because the Dirichlet energy density presents different growth in $B$ and $A$ and, hence, we provide a proof in the following.

\begin{thm}
Let $1<p, q<+\infty$, $p\neq q$ and $f\in X_\diamond^p(\partial \Omega)$. If (A1), (A2),  {(A3)} hold, then there exists a unique solution of problem \eqref{gminimum}.
\end{thm}
\begin{proof}
Before distinguishing the two cases depending on the exponents order, we observe that there exists a function $u_0\in W^{1,p}(\Omega)$ that assumes a suitable constant value in $A$, with $Tr(u_0)=f$ on $\partial\Omega$, such that $||u_0||_{W^{1,p}(\Omega)}\leq C(\partial\Omega)  ||f||_{X^p_\diamond(\partial\Omega)}<+\infty$, by the Inverse Trace inequality in Besov spaces \cite[Th. 18.34]{leoni17}.

For this function $u_0$, it is easily seen that $\mathbb E_\sigma(u_0)<+\infty$. As a consequence, $\mathbb E_\sigma$ is proper convex function, as required to apply \cite[Th. 3.30]{dacorogna2007direct}.

Moreover, the strictly convex function $Q_\sigma(x,\cdot)$ is coercive for a.e. $x\in\Omega$ with respect to $\min\{p,q\}$. Indeed, if $p>q$ and $E\ge E_0$, then, by assumption (A3), we have
\[
\begin{split}
&Q_B(x, E)\ge \underline{Q}\left(\frac{E}{E_0}\right)^p\ge \underline{Q}\left(\frac{E}{E_0}\right)^q\ge \underline{Q}\left[\left(\frac{E}{E_0}\right)^q-1\right]\quad \text{a.e. in } B;\\
&Q_A(x, E)\ge \underline{Q} \left(\frac{E}{E_0}\right)^q\ge \underline{Q}\left[\left(\frac{E}{E_0}\right)^q-1\right] \qquad\qquad\qquad\ \  \text{a.e. in } A.
\end{split}
\]
If $E<E_0$, then 
\[
\begin{split}
&Q_B(x, E)\ge 0
\ge\underline{Q}\left[\left(\frac{E}{E_0}\right)^q-1\right]\quad \text{a.e. in } B;\\
& Q_A(x, E)\ge 0 
\ge\underline{Q}\left[\left(\frac{E}{E_0}\right)^q-1\right]\quad \text{a.e. in } A.
\end{split}
\]
Therefore, setting $Q_\sigma=Q_B$ in $B$ and $Q_\sigma=Q_A$ in A, we have
\[
Q_\sigma(x, E)\ge \underline{Q}\left[\left(\frac{E}{E_0}\right)^q-1\right]\quad \text{a.e. in }\Omega,
\]
for any $E\ge 0$.

Similarly, when $p<q$, we have $Q_\sigma(x,E)\ge \underline{Q}\left[\left(\frac{E}{E_0}\right)^p-1\right]$, for any $E\ge0$.

Therefore, in both cases ($p>q$ and $p<q$), all the assumptions of \cite[Th. 3.30]{dacorogna2007direct} are satisfied and, thus, the solution exists and is unique.
\end{proof}

\begin{rem}
    By invoking \cite[Th. 3.30]{dacorogna2007direct} one finds that the solution is an element of $W^{1,\min\{p,q\}}(\Omega)=W^{1,p}(\Omega)\cup W^{1,q}(\Omega)$. In both cases ($p>q$ and $p<q$) the boundary data $f\in X_\diamond^p(\partial \Omega)$ is compatible with the solution space.
\end{rem}

Let us observe that optimization problems on domains with holes have received a great deal of interest in recent years, see e.g. \cite{della2020optimal,gavitone2021isoperimetric,paoli2020stability,paoli2020sharp} and references therein.

\subsection{Normalized solution}
Through this paper we study the behaviour of the solution of problem \eqref{gminimum} for small Dirichlet boundary data, i.e. the behaviour of $u^\lambda$ defined as:
\begin{equation}
    \label{Fspezzata}
\min_{\substack{u\in W^{1,p}(\Omega)\cup W^{1,q}(\Omega)\\ u=\lambda f\ \text{on}\ \partial \Omega}}\mathbb E_\sigma(u),
\end{equation}
for $\lambda\to 0^+$ (small Dirichlet data).

To this purpose, as discussed in Section \ref{underlying}, it is convenient to introduce the normalized solution $v^\lambda$ defined as:
\begin{equation*}
v^\lambda=\frac{u^\lambda}{\lambda}.
\end{equation*} 

In the following Sections, we prove that the behaviour of the normalized function $v^\lambda$ is $p_0-$Laplace modeled for $\lambda \to 0^+$.

For any prescribed $f\in X^p_\diamond(\partial \Omega)$ and $\lambda>0$, $v^\lambda$ is the solution of the following variational problem:
\begin{equation}\label{G_norm}
\min_{\substack{v\in W^{1,p}(\Omega)\cup W^{1,q}(\Omega)\\ v=f\ \text{on}\ \partial \Omega}}\mathbb G_0^\lambda(v),\quad \mathbb G_0^\lambda(v)=\frac 1 {\lambda^{p_0}}\left(\int_{B} Q_B(x,\lambda|\nabla v(x)|)dx+\int_{A} Q_A(x,\lambda|\nabla v(x)|)dx\right).
\end{equation}
The multiplicative factor $1/\lambda^{p_0}$ is introduced in order to guarantee that the functionals  {$\mathbb G_0^\lambda$} are equibounded for small $\lambda$. The normalized solution makes it possible to \lq\lq transfer\rq\rq\ parameter $\lambda$ in \eqref{Fspezzata} from the boundary data to the functional $\mathbb G_0^\lambda$.

Specifically, in the following Sections, we will prove that $v^\lambda$ converges, under very mild hypotheses, for $\lambda\to 0^+$. 

 {If $q_0<p_0$, the limiting problem of \eqref{G_norm} is a problem where the inner region $A$ is replaced by a PEC. The limit solution is termed $w^0$.}

 {If $p_0<q_0$,  the limiting problem of \eqref{G_norm} is a problem where the inner region $A$ is replaced by a PEI. The limiting problem of $v^\lambda$, is termed $v^0_B$ in $B$
.
}

 {Finally, we remark that $v^0_B$ and $w^0$ arise from a weighted $p_0-$Laplace problem
.}

\section{The fundamental inequality for small Dirichlet data}
\label{mean_sec0}
In this Section we provide the main tool to achieve the convergence results in the limiting cases for \lq\lq small\rq\rq\ Dirichlet boundary data. Specifically, we show that the asymptotic behaviour of the Dirichlet energy corresponds to a $p_0-$Laplace modelled equation \cite{Salo2012_IP,guo2016inverse} in domain $B$.

In the following, we study the asymptotic behaviour of the Dirichlet energy in the outer region $B$. To do this, we prove the following general Lemma, first for a  {weighted} $p_0-$Laplace problem and, then, for the quasilinear case.

Let $Q_F(x,E)=\theta(x)E^{p_0}$ be the Dirichlet energy density for a  {weighted} $p_0-$Laplace problem defined in $F$, a bounded Lipschitz domain. We observe that $Q_F$ satisfies assumption (A4).

\begin{lem}
\label{factorizable_lemma3}
Let $1<p_0<+\infty$, $F
\subset\R^n$ be a bounded domain with Lipschitz boundary, $\theta$ be a nonnegative measurable function  in  $F$ and 
$\{w_n\}_{n\in\N}$ be 
a sequence weakly convergent  to $w$ in $W^{1,p_0}(F)$. 
Then we have  \begin{equation}
\label{only_in_factor_E3}
\int_{F} \theta(x)|\nabla w(x)|^{p_0} dx\le\liminf_{n\to+\infty}\int_{F} \theta(x)|\nabla w_n(x)|^{p_0} dx.
\end{equation}
\end{lem}
\begin{proof}
Let us set $L:=\liminf_{n\to+\infty}\int_{F} \theta(x)|\nabla w_n(x)|^{p_0} dx$.  {If $L=+\infty$, the inequality \eqref{only_in_factor_E3} is trivial; otherwise we} consider a subsequence  {$\{n_j\}_{n\in\N}$} such that
\[
\lim_{ j\to+\infty}\int_{F} \theta(x)|\nabla w_{ {n_j}}(x)|^{p_0} dx=L.
\]
This means that for any $\varepsilon>0$, there exists $\nu\in\N$ such that
\begin{equation}
    \label{defin_liminf}
L-\varepsilon<\int_{F} \theta(x)|\nabla w_{ {n_j}}(x)|^{p_0} dx<L+\varepsilon
\end{equation}
for any $ j\ge\nu$. 
Then, by the Mazur's Lemma (refer for example to
\cite{dunford1963linear,renardy2006introduction}), there exists a function $N:\N\to\N$ and a sequence $\{\alpha_{n,k}\}_{k=n}^{N(n)}$ for any $n\in\N$ such that
\begin{enumerate}
    \item[(M1)] $\alpha_{n,k}\geq 0$\ for any $(n,k)\in\N\times[n,N(n)]$, 
    \item[(M2)] $\sum_{k=n}^{N(n)}\alpha_{n,k}=1$\ for any $n\in\N$,
    \item[(M3)] $z_n:=\sum_{k=n}^{N(n)}\alpha_{n,k}w_{ {n_k}}$ $\to w$ in $W^{1,p_0}(F)$.
\end{enumerate}
Then there exists a subsequence $\{z_{ {n_l}}\}_{ l\in\N}$ such that
\begin{equation}
    \label{eguagl_liminf}
\liminf_{n\to+\infty} \int_{F} \theta(x)|\nabla z_n(x)|^{p_0} dx=
\lim_{ l\to+\infty}\int_{F} \theta(x)|\nabla z_{ {n_l}}(x)|^{p_0} dx
\end{equation}
and another subsequence, again indicated with $\{z_{ {n_l}}\}_{ l\in\N}$,
such that $\nabla z_{ {n_l}}\to \nabla w$ a.e. in  {$F$} \cite[Chap. 18]{leoni17}. Therefore, we have
\begin{equation*}
    \begin{split}
\mathbb \int_{F} \theta(x)|\nabla w(x)|^{p_0} dx&=  \int_{F} \theta(x)\liminf_{ l\to+\infty}|\nabla z_{ {n_l}}(x)|^{p_0} dx\leq
\liminf_{ {l}\to+\infty}\int_{F} \theta(x)|\nabla z_{ {n_l}}(x)|^{p_0} dx\\
&=\lim_{ l\to+\infty}\int_{F} \theta(x)|\nabla z_{ {n_l}}(x)|^{p_0} dx =\liminf_{n\to+\infty}\int_{F} \theta(x)|\nabla z_n(x)|^{p_0} dx\\
&\le\liminf_{n\to+\infty}\sum_{k=n}^{N(n)}\alpha_{n,k}\int_{F} \theta(x)|\nabla w_{ {n_k}}(x)|^{p_0} dx\\
&<\liminf_{n\to+\infty}\mathbb   \sum_{k=n}^{N(n)}\alpha_{n,k}\left(L+\varepsilon\right)=L+\varepsilon\\
&=\liminf_{n\to+\infty}\int_{F} \theta(x)|\nabla w_n(x)|^{p_0} dx+\varepsilon,
\end{split}
\end{equation*}
where in the first line the equality follows from the convergence result of (M3) and the inequality follows from Fatou's Lemma, in the second line we applied \eqref{eguagl_liminf}, in the third line we applied the convexity of $|\cdot|^{p_0}$, and in the fourth line we applied \eqref{defin_liminf}.

Conclusion \eqref{only_in_factor_E3} follows from the arbitrariness of $\varepsilon>0$.
\end{proof}

The next step consists in extending \eqref{only_in_factor_E3} from the  {weighted} $p_0-$Laplace case to the quasilinear case. In doing this, we restrict the validity of the result to sequences of the solutions of problem \eqref{G_norm}. The main difficulty in proving this result lies in evaluating an upper bound of the measure of that part of $B$ where the solutions  {$v^\lambda$} admit large values of the gradient (see Figure \ref{fig_7_insiemi}).
\begin{figure}[ht]
    \centering
    \includegraphics[width=\textwidth]{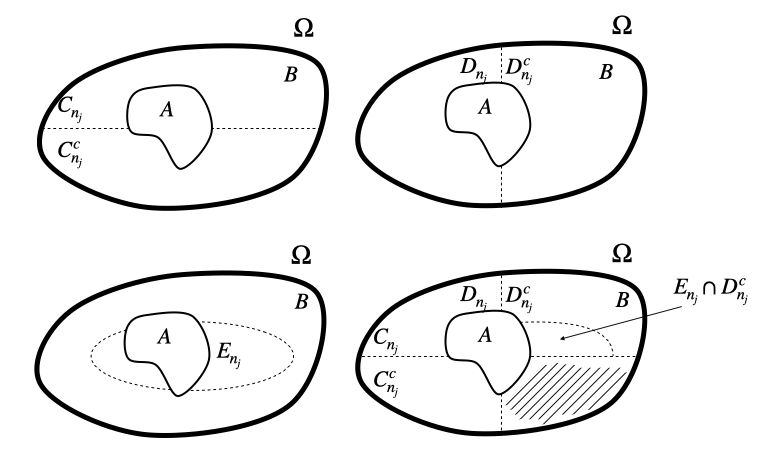}
    \caption{The objective of the proof is to show that the set of the points in $B$ such that the solution $v^{\lambda_{n_j}}$ does not satisfy the fundamental inequality ($C^c_{\delta,n_j}$) and admit large values of the gradient ($D^c_{\ell,n_j}$) can be made sufficiently small (shaded region).}
    \label{fig_7_insiemi}
\end{figure}

\begin{lem}
\label{lem_dis_fund_0}
Let $1<p_0\leq p<+\infty$, $f\in X^p_\diamond(\partial \Omega)$, (A1), (A2), (A3), (A4) hold, and let the solution $v^\lambda$ of \eqref{G_norm} be weakly convergent to $v$ in $W^{1,p_0}(B)$, for $\lambda\to 0^+$. Let $\lambda_n\to 0^+$ be a decreasing sequence for $n\to+\infty$, such that
\begin{equation}
    \label{seq_to_liminf_0}
\lim_{n\to+\infty}\frac{1}{\lambda_n^{p_0}}\int_{B}Q_B(x,\lambda_n|\nabla v^{\lambda_n}(x)|)dx=\liminf_{\lambda\to0^+}\frac{1}{\lambda^{p_0}}\int_{B}Q_B(x,\lambda|\nabla v^{\lambda}(x)|)dx.
\end{equation}
Then, for any $\delta>0$ and $\theta>0$, there exists a set $F_{\delta,\theta}\subseteq B$ with $|B\setminus F_{\delta,\theta}|<\theta$ such that
\begin{equation}
\label{step_to_conclude_0}
\liminf_{n\to+ \infty}\int_{F_{\delta,\theta}} (\beta_0(x)-\delta)|\nabla v^{\lambda_n}(x)|^{p_0} dx\le \lim_{n\to+\infty}\frac{1}{\lambda_n^{p_0}}\int_{B}Q_B(x,\lambda_n|\nabla v^{\lambda_n}(x)|)dx.
\end{equation}
\end{lem}
\begin{proof}
Let $ {w^0}$ be the solution of problem  {\eqref{pproblem_Bgrad0}}. We have
\begin{equation*}
\begin{split}
\frac{\underline{Q}}{E_0^{p_0}}\int_{B}|\nabla v^\lambda(x)|^{p_0}dx&\leq\frac 1 {\lambda^{p_0}}\int_{B} Q_B(x,\lambda|\nabla v^\lambda(x)|)dx\leq \mathbb{G}_0^\lambda(v^\lambda)\leq\mathbb{G}_0^\lambda( {w^0})\\
&=\frac 1 {\lambda^{p_0}}\int_{B} Q_B(x,\lambda|\nabla  {w^0}(x)|)dx\\
&\leq\max\left\{\frac{\overline Q}{E_0^{p_0}}\int_B|\nabla  {w^0}(x)|^{p_0}dx , \frac{\overline Q}{E_0^p}\lambda^{p-p_0}\int_B|\nabla  {w^0}(x)|^pdx\right\},
\end{split}
\end{equation*}
where in the first inequality we used the left-hand side of (A3.i), in the second inequality we added the integral term on $A$, in the third inequality we tested $\mathbb G_0^\lambda$ with $ {w^0}$, and in the last inequality we used the right-hand side (A3.i).
 
Since $p-p_0\geq 0$, we find that
$\int_{B}|\nabla v^\lambda(x)|^{p_0}dx$ is definitively upper bounded and, therefore, $
\int_{B} |\nabla v^{\lambda_n}(x)|^{p_0}dx$ is upper bounded by a constant $M>0$, for any $n\in\N$.

Let us fix  $\delta>0$ and $\theta>0$. For any $n\in\N$, we set
\begin{equation}
    \label{def_C0}
C_{\delta,n}=C_{n}:=\left\{x\in B\ : \ (\beta_0(x)-\delta)(\lambda_n|\nabla v^{\lambda_n}(x)|)^{p_0}\le Q_B(x,\lambda_n|\nabla v^{\lambda_n}(x)|)\right\}.
\end{equation}
Now, for any constant $L>0$, we set
\begin{equation*}
\begin{split}
D_{n,L}=D_{n}:=\{ x\in B \ : \ |\nabla v^{\lambda_n}(x)|< L\}\\
D^c_{n,L}=D^c_{n}:=\{ x\in B \ : \ |\nabla v^{\lambda_n}(x)|\ge  L\}
\end{split}
\end{equation*}
By definition of $D^c_{n}$, we have
\[
|D^c_{n}|L^{p_0}\le
\int_{B} |\nabla v^{\lambda_n}(x)|^{p_0}dx
\le M,
\]
which gives $|D^c_n|\leq \frac{M}{
L^{p_0}}$, for any $n\in\N$. By choosing $L>\left(\frac{4M}{\theta}\right)^\frac{1}{p_0}$, we have
\begin{equation}
\label{stimaD3_0}
|D^c_n|<\frac\theta 4.
\end{equation}
Let $E_{n}$ be defined as
\[
E_{\delta,L,n}=E_{n}:=\left\{x\in B\ : \ (\beta_0(x)-\delta)E^{p_0}\leq {Q_B(x,E)}\ \  \forall\ 0\leq E<\lambda_nL\right\}.
\]
Then, for any $n\in\N$,
\begin{equation}
\label{chain_inclusion3_0}
C_{n}\supseteq C_{n}\cap D_{n}\supseteq E_{n}\cap D_{n}.
\end{equation}

$E_n$ is increasing with respect to $n$ and $\left| \bigcup_{n=1}^{+\infty}E_{n}\right|=|B|$. Therefore there exists a natural number $n_1=n_1(\theta)$ such that
\begin{equation}
    \label{stimaF3_0}
\left|\bigcup_{n=1}^{n_1} E_{n}\right|=\left|E_{n_1}\right|\ge |B|-\frac \theta 4.
\end{equation}
By considering the complementary sets in $B$, \eqref{chain_inclusion3_0} for $n=n_1$ gives
\[
C^c_{n_1}\subseteq E^c_{n_1}\cup D^c_{n_1}=E^c_{n_1}\cup D^c_{n_1}\quad
\]
with
\[
\quad |C_{n_1}^c|\leq \frac\theta 4+\frac\theta 4=\frac\theta 2,
\]
as follows from \eqref{stimaD3_0} and \eqref{stimaF3_0}. Similarly,  {repeating the argument and replacing $2^{j+1}$ with $4$ in the previous argument}, we construct a subsequence $\{\lambda_{n_j}\}_{j\in\N}$ such that
\[
|C_{n_j}^c|\leq \frac{\theta}{2^j}.
\]
Then, by defining
\[
F_{\delta,\theta}: =\bigcap_{j=1}^\infty C_{n_j}, 
\]
we have
\[
F^c_{\delta,\theta}= \bigcup_{j=1}^\infty C_{n_j}^c
\quad
\]
with $|F^c_{\delta,\theta}|\leq \theta\sum_{j=1}^{+\infty} {2^{-j}}=\theta$, which means
\[
|B |-\theta\le |F_{\delta,\theta}|\le |B|.
\]
Therefore, we have
\begin{equation*}
\begin{split}
&\liminf_{j\to+\infty}\int_{F_{\delta,\theta}}(\beta_0(x)-\delta)|\nabla v^{\lambda_{n_j}}(x)|^{p_0} dx\le \liminf_{j\to+\infty}\frac{1}{\lambda_{n_j}^{p_0}}\int_{F_{\delta,\theta}} Q_B(x,|\lambda_{n_j}\nabla v^{\lambda_{n_j}}(x)|)dx\\
&\le\liminf_{j\to +\infty}\frac{1}{\lambda_{n_j}^{p_0}}\int_{B} Q_B(x,|\lambda_{n_j}\nabla v^{\lambda_{n_j}}(x)|)dx=\lim_{n\to+\infty}\frac{1}{\lambda_n^{p_0}}\int_{B}Q_B(x,\lambda_n|\nabla v^{\lambda_n}(x)|)dx,
\end{split}
\end{equation*}
where in the first inequality we use $F_{\delta,\theta}\subseteq C_{n_j}$ for any $j\in\N$ and the inequality appearing in \eqref{def_C0}, in the second inequality we take into account that $F_{\delta,\theta}\subseteq B$, and, in the last equality, we exploit the convergence given by \eqref{seq_to_liminf_0}.
\end{proof}

Finally, we prove the result on the fundamental inequality holding for $\lambda\to 0^+$.
\begin{prop}
\label{fund_ine_propt0}
Let $1<p_0\le p<+\infty$, $f\in X^p_\diamond(\partial \Omega)$, (A1), (A2), (A3), (A4) hold, and let the solution $v^\lambda$ of \eqref{G_norm} be weakly convergent to $v$ in $W^{1,p_0}(B)$, as $\lambda\to 0^+$. Then \begin{equation}\label{fundamental_inequality3_0}
\int_{B}\beta_0(x)|\nabla v|^{p_0}dx\le
\liminf_{\lambda\to 0^+}\frac 1 {\lambda^{p_0}}\int_{B} Q_B(x,\lambda|\nabla v^\lambda(x)|)dx.
\end{equation}
\end{prop}
\begin{proof}
First, we assume that $v$ is nonconstant in $B$, otherwise the conclusion is trivial. Therefore, the integral on the l.h.s. of \eqref{fundamental_inequality3_0} is positive because $\beta_0\in L^\infty_+(B)$.

Let $\{\lambda_n\}_{n \in \N}$ be a decreasing sequence such that $\lambda_n\to 0^+$ for $n\to+\infty$, and
\begin{equation}
    \label{seq_to_liminf3_0}
\lim_{n\to+\infty}\frac{1}{\lambda_n^{p_0}}\int_{B}Q_B(x,\lambda_n|\nabla v^{\lambda_n}(x)|)dx=\liminf_{\lambda\to0^+}\frac{1}{\lambda^{p_0}}\int_{B}Q_B(x,\lambda|\nabla v^{\lambda}(x)|)dx.
\end{equation}
To prove \eqref{fundamental_inequality3_0}, we use Lemma \ref{lem_dis_fund_0}. To this purpose, the measure
\begin{equation*}
\kappa: F\in  {\mathcal B(\Omega)}\mapsto \int_F\beta_0(x)|\nabla v(x)|^{p_0}dx
\end{equation*}
 {where $\mathcal B(\Omega)$ is the class of the borelian sets contained in $\Omega$,}
is absolutely continuous with respect to the Lebesgue measure. Therefore, for any $\varepsilon>0$, there exists $\theta>0$ such that $\kappa(F)<\frac\varepsilon 2$ for any $|F|<\theta$.
Therefore,  {for any fixed $0<\delta\le\inf_{B}\beta_0$},
$|B\setminus F_{\delta,\theta}|<\theta$ implies 
\begin{equation}
        \label{mu3_0}
\kappa(B\setminus F_{\delta,\theta})=\int_{B\setminus F_{\delta,\theta}}\beta_0(x)|\nabla v(x)|^{p_0}dx <\frac \varepsilon 2.
\end{equation}
Hence, for  $\delta\le \min\left\{\varepsilon/\left(2\int_{B}|\nabla v(x)|^{p_0} dx\right), \inf_B \beta_0(x)\right\}$, we have
\begin{equation*}
\begin{split}
&\int_{B} \beta_0(x)|\nabla v(x)|^{p_0} dx-\varepsilon<\int_{B} \beta_0(x)|\nabla v(x)|^{p_0} dx-\int_{B\setminus F_{\delta,\theta}} \beta_0(x)|\nabla v(x)|^{p_0} dx-\frac \varepsilon 2\\
&\le\int_{F_{\delta,\theta}} \beta_0(x)|\nabla v(x)|^{p_0} dx- \delta \int_{B} |\nabla v(x)|^{p_0} dx\le \int_{F_{\delta,\theta}} (\beta_0(x)-\delta)|\nabla v(x)|^{p_0} dx\\
&\le \liminf_{\lambda\to 0^+}\int_{F_{\delta,\theta}} (\beta_0(x)-\delta)|\nabla v^{\lambda}(x)|^{p_0} dx \le\lim_{n\to+\infty}\frac{1}{\lambda_n^{p_0}}\int_{B}Q_B(x,\lambda_n|\nabla v^{\lambda_n}(x)|)dx\\
&=\liminf_{\lambda\to 0^+}\frac{1}{\lambda^{p_0}}\int_{B}Q_B(x,\lambda|\nabla v^{\lambda}(x)|)dx,
\end{split}
\end{equation*}
where in the first line we applied \eqref{mu3_0}, in the second line the connection between $\delta$ and $\varepsilon$, in the third line we used Lemma \ref{factorizable_lemma3} with $F= F_{\delta,\theta}$ and the inequality \eqref{step_to_conclude_0} of the previous Lemma \ref{lem_dis_fund_0}, and in the fourth line we used \eqref{seq_to_liminf3_0}.


The conclusion follows from the arbitrariness of $\varepsilon$.
\end{proof}

\section{Limiting Problems for small Dirichlet data}\label{small_sec}

In this Section we treat the limiting case of problem \eqref{Fspezzata} for small Dirichlet boundary data. 

We will distinguish two cases depending on $p_0$ and $q_0$:
\begin{enumerate}
    \item[(i)] $1<q_0<p_0<+\infty$ (see Section \ref{Small_bigger});
    \item[(ii)] $1<p_0<q_0<+\infty$ (see Section \ref{Small_smaller}).
\end{enumerate}
In the first case, we prove that: 

\begin{itemize}
\item[(i.a)] $v^\lambda\rightharpoonup w
^0$ in $W^{1,p_0}(B)$, as $\lambda\to 0^+$, where $w^0$ in $B$ is the unique solution of problem \eqref{pproblem_Bgrad0};
\item[(i.b)] $v^\lambda\to w^0$ in $W^{1,q_0}(A)$, as $\lambda\to 0^+$, where $w^0$ is constant in any connected component of $A$;
\item[(i.c)] $v^\lambda\rightharpoonup w
^0$ in $W^{1,p_0}(\Omega)$, as $\lambda\to 0^+$, where $w^0$ in $B$ is the unique solution of problem \eqref{pproblem_Bgrad0} and in $A$ is constant in any connected component.
\end{itemize}
The limiting solution in $\Omega$, for $1<q_0<p_0<+\infty$, is characterized by:
\begin{equation}
\label{H}
\min_{\substack{v\in W^{1,p_0}(\Omega)\\ |\nabla v|=0\ \text{a.e. in}\ A \\ v=f\ \text{on}\ \partial \Omega}}\mathbb B_0(v),\quad \mathbb B_0(v)=\int_{B} \beta_0(x)|\nabla v(x)|^{p_0} dx.
\end{equation}
Problem \eqref{H} is the variational form of problem \eqref{pproblem_Bgrad0}.

Whereas, in the second case, we prove that:
\begin{itemize}
\item[(ii
)] $v^\lambda\rightharpoonup 
v^0_B$ in $W^{1,p_0}(B)$, as $\lambda\to 0^+$, where $v^0_B$ in $B$ is the unique solution of problem \eqref{pproblem_B0}.
\end{itemize}

For $1<p_0<q_0<+\infty$, the limiting solution is characterized by the following problem:
\begin{equation}
    \label{Hii}
\min_{\substack{v\in W^{1,p_0}(B)\\ v=f\ \text{on}\ \partial \Omega}}\mathbb B_0(v),\quad \mathbb B_0(v)=\int_{B} \beta_0(x)|\nabla v(x)|^{p_0}dx
\end{equation}
The problem \eqref{Hii} is the variational form of \eqref{pproblem_B0}.
 {We recall that $v_B^0\in W^{1,p_0}(B)$ 
is the unique normalized solutions of the limiting problems  {\eqref{Hii}
} in region $B$
}.

Using the results developed in Section \ref{mean_sec0}, we are in a position to prove the main convergence results.
\subsection{First case: $\mathbf{q_0<p_0}$
}
\label{Small_bigger}
 {In the whole section we assume $q_0< p_0\le p$ and $q_0\leq q$; hence we have the continuous embeddings $W^{1,p}(\cdot)\hookrightarrow W^{1,p_0}(\cdot)\hookrightarrow W^{1,q_0}(\cdot)$ and $W^{1,q}(\cdot)\hookrightarrow W^{1,q_0}(\cdot)$, for any bounded set with Lipschitz boundary.}

For any fixed $f\in X^{p}_\diamond(\partial \Omega)$ we study problem \eqref{G_norm} as $\lambda$ approaches zero. The variational problem  \eqref{G_norm} particularizes as
\begin{equation}
\label{G^t}
\min_{\substack{v\in W^{1,q_0}(\Omega)\\ v=f\ \text{on}\ \partial \Omega}}\mathbb G_0^\lambda(v),\quad \mathbb G_0^\lambda(v)=\frac 1 {\lambda^{p_0}}\left(\int_{B} Q_B(x,\lambda|\nabla v(x)|)dx+\int_{A} Q_A(x,\lambda|\nabla v(x)|)dx\right).
\end{equation}

Let $v^\lambda$ be the minimizer of \eqref{G^t} and $ {w^0}$ be the minimizer of \eqref{H}; the aim of this Section is to prove the following convergence result 
\begin{equation*} 
v^\lambda\rightharpoonup 
 {w^0}\quad \text{in}\ W^{1,q_0}(\Omega) \quad\text{as}\ \lambda\to 0^+.
\end{equation*}

The condition $|\nabla v|=0$ is equivalent to saying that $v$ is constant on each connected component of $A$. This makes it possible to decouple the problems associated to regions $B$ and $A$. Specifically, region $A$ behaves as a PEC, with respect to problem \eqref{H}, whereas the outer region $B$ behaves as $p-$Laplacian modelled material with a PEC on $\partial A$.

 {
We first prove that $v^\lambda$ is weakly convergent and we identify the limiting function $v^0\in W^{1,q_0}(\Omega)$, for $\lambda \to 0^+$; proving that the limiting function $v^0$ satisfies  $\mathbb B_0(v^0)=\mathbb B_0(w^0)$}. The latter equality implies $v^0= {w^0}$, because of the uniqueness of the solution of problem \eqref{H}. Then, assuming stronger hypotheses, we prove the strong convergence in $W^{1,q_0}(\Omega)$.

\begin{thm}\label{Thm_conv_lim}
Let $1<q_0<p_0<+\infty$ be such that $p_0\leq p$, $q_0\leq q$, $f\in X^{p}_\diamond(\partial \Omega)$ and $v^\lambda$ be the solution of \eqref{G^t}. If (A1), (A2), (A3) and (A4) hold, then
\begin{itemize}
\item[(i)] $v^\lambda\rightharpoonup  {w^0}$ in $W^{1,p_0}(B)$, as $\lambda\to 0^+$,
\item[(ii)] $v^\lambda\to  {w^0}$ in $W^{1,q_0}(A)$, as $\lambda\to 0^+$,
\item[(iii)] $v^\lambda\rightharpoonup  {w^0}$ in $W^{1,q_0}(\Omega)$, as $\lambda\to 0^+$,
\end{itemize}
where $ {w^0}\in W^{1,p_0}(\Omega)$
is the unique solution of
\eqref{H}.
\end{thm}
\begin{proof}
For the sake of simplicity, we will only treat the case when $A$ has one connected component. The general case can be treated with the same approach. 

 {
Let us consider a function in $W^{1,p}(\Omega)$ whose trace on $\partial\Omega$ is $f$ and is such that $f\equiv w^0$ in $A$. We again denote this function by $f$ and hence we have $w^0-f\in W^{1,p_0}_0(B)$.}

 {
A density argument \cite[Th. 11.35]{leoni17} ensures that there exists $\{v_n\}_{n\in\N}\subseteq C_c^\infty(B)$ such that
\[
v_n \to w_0-f\quad \text{in } W^{1,p_0}(\Omega), \text{ as } n\to\infty.
\]
Consequently, we have
\[
\lim_{n\to+\infty}
\int_\Omega |\nabla v_n - \nabla (w^0-f)|^{p_0}dx=0.
\]
We immediately deduce that $f+v_n\in W^{1,p}(B)$ and
\[
\lim_{n\to+\infty}
\int_\Omega \beta_0(x)|\nabla (v_n+f) - \nabla w^0|^{p_0}dx=0.
\]
Hence, this implies that
\[
\lim_{n\to+\infty}
\mathbb B_0(v_n+f) =\mathbb B_0( w^0).
\]
}
 {
Therefore, for any $\varepsilon>0$, there exists $\omega\in W^{1,p}(\Omega)$ with $Tr(\omega)=f$ on $\partial\Omega$ and with constant value on $A$ such that:
\begin{equation}
\label{lavrentiev}    
\mathbb B_0(\omega)<\mathbb B_0(w^0)+\varepsilon.
\end{equation}
} 
Hence,
\begin{equation}
\label{chainGtvtv0Omega}
\begin{split}
\frac{ \underline{Q}}{E_0^{p_0}} \int_{B}|\nabla v^\lambda(x)|^{p_0}dx& \leq\frac 1 {\lambda^{p_0}}\int_{B} Q_B(x,\lambda|\nabla v^\lambda(x)|)dx\\
&\leq\mathbb G_0^\lambda (v^\lambda)\le \mathbb G_0^\lambda ( {\omega})=\frac 1 {\lambda^{p_0}}\int_{B} Q_B(x,|\lambda\nabla  {\omega}(x)|)dx\\
&\le\max\left\{ \frac{ \overline{Q}}{E_0^{p_0}} \int_{B}|\nabla  {\omega}(x)|^{p_0}dx,\frac{ \overline{Q}}{E_0^{p}}\lambda^{p-p_0} \int_{B}|\nabla  {\omega}(x)|^{p}dx\right\},
    \end{split}
\end{equation}
where in the first inequality we used (A3.i)-left, in the second inequality we exploited the fact that $\mathbb G^\lambda_0$ also contains the integral term over $A$, in the third inequality we used the fact that $v^\lambda$ is the minimizer of $\mathbb G_0^\lambda$, in the last inequality we used (A3.i)-right.

Since $\lambda^{p-p_0}$ is bounded, as $\lambda\to0^+$, we find that $\int_{B}|\nabla v^\lambda(x)|^{p_0}dx$ is definitively upper bounded by \eqref{chainGtvtv0Omega}. 

 {We first prove that $\{v^\lambda\}_\lambda\subseteq L^{p_0}(B)$ is equibounded. 
By using the Poincaré inequality \cite[Th. 13.19]{leoni17}, we have:
\begin{equation}
\label{bounded_poincare_vl}
\begin{split}
||v^\lambda||_{L^{p_0}(B)}&\leq  ||v^\lambda- f||_{L^{p_0}(B)}+|| f ||_{L^{p_0}(B)}\leq C||\nabla v^\lambda-\nabla f||_{L^{p_0}}+||f||_{L^{p_0}(B)}\\
&\leq C||\nabla v^\lambda||_{L^{p_0}(B)}+C||\nabla f||_{L^{p_0}(B)}+||f||_{L^{p_0}(B)}.
\end{split}
\end{equation}
The claim is proved because the right hand side is equibounded.}

 {Taking into account \eqref{chainGtvtv0Omega} and \eqref{bounded_poincare_vl},  
it follows that there exists $v^0\in W^{1,{p_0}}(B)$} such that, up to a subsequence, $v^\lambda\rightharpoonup v^0$ in $W^{1,{p_0}}(B)$, as $\lambda\to 0^+$, that is {\it (i)}.

Similarly to \eqref{chainGtvtv0Omega}, for region $A$, we have
\begin{equation}
\label{chainGtvtv0A}
\begin{split}
\frac{ \underline{Q}}{E_0^{q_0}}\int_{A}|\nabla v^\lambda(x)|^{q_0}dx & \leq\frac 1 {\lambda^{q_0}}\int_{A} Q_A(x,\lambda|\nabla v^\lambda(x)|)dx=\frac 1 {\lambda^{q_0-p_0}}\frac 1 {\lambda^{p_0}}\int_{A} Q_A(x,\lambda|\nabla v^\lambda(x)|)dx\\
&\leq\frac{1}{\lambda^{q_0-p_0}}\mathbb G_0^\lambda (v^\lambda)\le \frac{1}{\lambda^{q_0-p_0}}\mathbb G_0^\lambda ( {\omega})=\frac{1}{\lambda^{q_0-p_0}}\frac 1 {\lambda^{p_0}}\int_{B} Q_B(x,|\lambda\nabla  {\omega}(x)|)dx\\
&\le\frac{1}{\lambda^{q_0-p_0}}\max\left\{ \frac{ \overline{Q}}{E_0^{p_0}} \int_{B}|\nabla  {\omega} (x)|^{p_0}dx,\frac{ \overline{Q}}{E_0^p}\lambda^{p-p_0} \int_{B}|\nabla  {\omega}(x)|^{p}dx\right\},   
\end{split}
\end{equation}
where we exploited (A3.ii)-left in the first inequality. Therefore, by passing \eqref{chainGtvtv0A} to the limit, we have\[\lim_{\lambda\to 0^+}\lambda^{q_0-p_0}\int_{A}|\nabla v^\lambda(x)|^{q_0}dx\leq\frac{\overline{Q}}{\underline Q} E_0^{q_0-p_0} \int_{B}|\nabla  {\omega}(x)|^{p_0}dx.\]
Therefore, we find that
$\int_{A}|\nabla v^\lambda(x)|^{q_0}dx=O(\lambda^{p_0-q_0})$; 
 {now, we prove that $\{v^\lambda\}_\lambda\subseteq L^{q_0}(A)$ is equibounded.
By using the Poincaré inequality \cite[Th. 13.19]{leoni17}, we have:
\begin{equation}
\label{bounded_poincare_vlq}
\begin{split}
||v^\lambda||_{L^{q_0}(A)}\leq ||v^\lambda||_{L^{q_0}(\Omega)}& \leq  ||v^\lambda- f||_{L^{q_0}(\Omega)}+|| f ||_{L^{q_0}(\Omega)}\\
&\leq C||\nabla v^\lambda-\nabla f||_{L^{q_0}(\Omega)}+||f||_{L^{q_0}(\Omega)}\\
&\leq C||\nabla v^\lambda||_{L^{q_0}(\Omega)}+C||\nabla f||_{L^{q_0}(\Omega)}+||f||_{L^{q_0}(\Omega)}.
\end{split}
\end{equation}
The claim is proved because the right hand side is equibounded. H}ence, $||v^\lambda||_{W^{1,q_0}(A)}$ is definitively upper bounded.

Moreover, by taking into account \eqref{chainGtvtv0Omega}, \eqref{chainGtvtv0A}, \eqref{bounded_poincare_vlq}, that $q_0<p_0\le p$, 
it turns out that $v^\lambda\in W^{1,q_0}(\Omega)
$ and that $||v^\lambda||_{W^{1,q_0}(\Omega)}$ is upper bounded.
Therefore, up to a subsequence, we find that $v^\lambda\rightharpoonup v^0$ in $W^{1,q_0}(\Omega)$. Moreover, $v^\lambda\to v^0$ in $W^{1,q_0}(A)$ and $v^0$ is constant in $A$ because $\nabla v^\lambda \to 0$ in $L^{q_0}( {A})$. We have thus proved convergences {\it (ii)} and {\it (iii)}. 

The final step is to prove that $v^\lambda$ converges to $ {w^0}$, which is the minimizer for \eqref{H}. Specifically, 
we have the following inequalities:
\begin{equation}\label{chain_small_i}
    \begin{split}
\mathbb B_0( {w^0})\le\mathbb B_0(v^0)&\le\liminf_{\lambda\to 0^+}\frac 1 {\lambda^{p_0}}\int_{B} Q_B(x,|\lambda\nabla v^\lambda(x)|)dx\le \liminf_{\lambda\to 0^+}\mathbb G_0^\lambda(v^\lambda)\\
&\le 
\lim_{\lambda\to 0^+}\mathbb G_0^\lambda( {\omega})= \mathbb B_0(\omega)<\mathbb B_0(w^0)+\varepsilon,
\end{split}
\end{equation}
where in the first inequality we exploited that $ {w^0}$ is the minimizer of $\mathbb B_0$, in the second inequality we used the fundamental inequality of Proposition \ref{fund_ine_propt0}, in the third inequality we added the integral term on region $A$, in the fourth inequality 
we exploited that $v^\lambda$ is the minimizer of $\mathbb G_0^\lambda$,  {in the equality in the second line we have taken into account assumption (A4) and the dominated convergence Theorem, and in the last inequality we have used \eqref{lavrentiev}}.

 {By the arbitrariness of $\varepsilon>0$, t}his implies that $\mathbb B_0( {w^0})=\mathbb B_0(v^0)$ and, hence, $v^0= {w^0}$ due to uniqueness of the solution of minimization problem \eqref{H}.
\end{proof}

\begin{rem}
\eqref{chain_small_i} implies the equality in the fundamental inequality \eqref{fundamental_inequality3_0}.
\end{rem}

\subsection{Second case: $\mathbf{p_0<q_0}$}
\label{Small_smaller}

 {In the whole section we assume $p_0< q_0\le q$ and $p_0\leq p$; hence we have the continuous embeddings $W^{1,q}(\cdot)\hookrightarrow W^{1,q_0}(\cdot)\hookrightarrow W^{1,p_0}(\cdot)$ and $W^{1,p}(\cdot)\hookrightarrow W^{1,p_0}(\cdot)$ on any bounded set with Lipschitz boundary.}

In this case, problem \eqref{G_norm} particularizes as 
\begin{equation}
\label{G^tii}
\min_{\substack{v\in W^{1,p_0}(\Omega)\\ v=f\ \text{on}\ \partial \Omega}}\mathbb G_0^\lambda(v),\quad \mathbb G_0^\lambda(v)=\frac 1 {\lambda^{p_0}}\left(\int_{B} Q_B(x,\lambda|\nabla v(x)|)dx+\int_{A} Q_A(x,\lambda|\nabla v(x)|)dx\right),
\end{equation}
for any prescribed $f\in X_\diamond^{p}(\partial \Omega)$.

In this case, limiting problem \eqref{Hii} plays a key role. Specifically, we have the following Theorem.

\begin{thm}\label{Thm_conv_limii}
Let $1<p_0<q_0<+\infty$ with $p_0\leq p$ and $q_0\leq q$, $f\in X^{p}_\diamond(\partial \Omega)$ and $v^\lambda$ be the solution of \eqref{G^tii}. If (A1), (A2), (A3) and (A4) hold, then
\begin{itemize}
\item[] {$v^\lambda\rightharpoonup v_B^0$ in $W^{1,p_0}(B)$, as $\lambda\to 0^+$},
\end{itemize}
where $v_B^0\in W^{1,p_0}(B)$
is the unique solution of
\eqref{Hii}.
\end{thm}
\begin{proof}
 {
Let us consider a Sobolev extension $\tilde v_B^0$ in $W^{1,p_0}(\Omega)$ of $v_B^0 \in W^{1,p_0}(B)$ and a function in $W^{1,p}(\Omega)$ whose trace on $\partial\Omega$ is $f$, that we again denote by $f$. Hence we observe that $\tilde v_B^0-f\in W^{1,p_0}_0(\Omega)$.}

 {A density argument \cite[Th. 11.35]{leoni17} ensures that there exists $\{v_n\}_{n\in\N}\subseteq C_c^\infty(\Omega)$ such that
\[
v_n \to \tilde v_B^0-f\quad \text{in } W^{1,p_0}(\Omega), \text{ as } n\to\infty.
\]
Consequently, we have
\[
\lim_{n\to+\infty}
\int_\Omega |\nabla v_n - \nabla (\tilde v_B^0-f)|^{p_0}dx=0.
\]
Therefore, we immediately deduce that $f+v_n\in W^{1,p}(\Omega)$ and
\[
\lim_{n\to+\infty}
\int_B \beta_0(x)|\nabla (v_n+f) - \nabla v_B^0|^{p_0}dx=0.
\]
Hence, this implies that
\[
\lim_{n\to+\infty}
\mathbb B_0(v_n+f) =\mathbb B_0( v_B^0).
\]
}
 {
Therefore, for any $\varepsilon>0$, there exists $\omega\in W^{1,p}(\Omega)$ with $Tr(\omega)=f$ on $\partial\Omega$ such that:
\begin{equation}
\label{lavrentiev2}    
\mathbb B_0(\omega)<\mathbb B_0(v_B^0)+\varepsilon.
\end{equation}
Hence, we have
\begin{equation}
\label{chainGtvtv0Omega2}
\begin{split}
\frac{ \underline{Q}}{E_0^{p_0}} \int_{B}|\nabla v^\lambda(x)|^{p_0}dx& \leq\frac 1 {\lambda^{p_0}}\int_{B} Q_B(x,\lambda|\nabla v^\lambda(x)|)dx\\
&\leq\mathbb G_0^\lambda (v^\lambda)\le \mathbb G_0^\lambda ( {\omega})=\frac 1 {\lambda^{p_0}}\int_{B} Q_B(x,|\lambda\nabla  {\omega}(x)|)dx\\
&\le\max\left\{ \frac{ \overline{Q}}{E_0^{p_0}} \int_{B}|\nabla  {\omega}(x)|^{p_0}dx,\frac{ \overline{Q}}{E_0^{p}}\lambda^{p-p_0} \int_{B}|\nabla  {\omega}(x)|^{p}dx\right\},
    \end{split}
\end{equation}
where in the first inequality we used (A3.i)-left, in the second inequality we exploited the fact that $\mathbb G^\lambda_0$ also contains the integral term over $A$, in the third inequality we used the fact that $v^\lambda$ is the minimizer of $\mathbb G_0^\lambda$ and in the last inequality we used (A3.i)-right.}


Since $||\nabla v^\lambda||_{L^{p_0}(B)}^{p_0}$ is definitively upper bounded and $v^\lambda=f$ on $\partial\Omega$, then  {by the Rellich–Kondrachov’s compactness Theorem \cite[Th. 12.18]{leoni17},} there exists a function $v^0\in W^{1,p_0}(B)$ such that 
\begin{equation}
    \label{conv_vt_v_0}
v^\lambda \rightharpoonup v^0\quad\text{in}\ W^{1,p_0}(B)\quad\text{as}\ \lambda\to 0^+.
\end{equation}

The final step is to prove that $v^0$ 
is equal to $v_B^0$, the solution of the limiting problem \eqref{Hii}. 

Let $\delta>0$ be prescribed, and let $A_\delta$ be the $\delta$-Minkowski neighbourhood of $A$: 
\[
A_\delta=\{x\in\Omega\ : \ \dist(x,A)<\delta\}.
\]
For any $\tau$ such that $0<\tau<\delta$, we denote the mollified function $({  {\omega}})_\tau=\rho_\tau * { {\omega}}$, where $\rho_\tau$ is the canonical mollifier. Then,  
for any $0<\tau<\delta$, we
define
\[
 {z}(x)=
\begin{cases}
 {\omega} &\text{in}\ B\setminus A_{\delta},\\
\frac{\dist(x,A_\delta)}\delta  {\omega}+\left(1-\frac{\dist(x,A_\delta)}\delta \right)( { \omega})_\tau & \text{in}\ A_{\delta}\setminus A,\\
( { \omega})_\tau &\text{in}\ A.
\end{cases}
\]

 {Let us observe that the mollification of $\omega$ in $A$ (in the definition of $z$) is well-defined, because $A\subset\subset \Omega$.}

We have the following inequalities:
\begin{equation}
\label{chainii}
\begin{split}
\mathbb B_0(v_B^0)&\le\mathbb B_0(v^0)\le \liminf_{\lambda\to 0^+}\mathbb G_0^\lambda(v^\lambda) 
\le\lim_{\lambda\to 0^+}\mathbb G_0^\lambda( {z})\\
&\leq\int_{B\setminus A_{\delta}}\beta_0(x)|\nabla  {\omega}(x)|^{p_0}dx\\
&\qquad+\frac{ \overline{Q}}{E_0^p} \left(\int_{A_{\delta}\setminus A} |\nabla  {\omega}(x)|^{p_0}+|\nabla (  {\omega})_\tau(x)|^{p_0}+\frac{|  {\omega}(x)-( {\omega})_\tau(x)|^{p_0}}{\delta^{p_0}}dx\right)\\
&\qquad+\frac{ \overline{Q}}{E_0^{q_0}}\lim_{\lambda\to 0^+} \lambda^{q_0-p_0}\int_A|\nabla ( {\omega})_\tau(x)|^{q_0}dx\\
&\leq \mathbb B_0( {\omega})+I_{\delta,\tau} {<\mathbb B_0(v_B^0)+I_{\delta,\tau}+\varepsilon},
\end{split}
\end{equation}
where in the first inequality we exploited the fact that $v^0_B$ is the minimizer of $\mathbb B_0$, in the second inequality we used the fundamental inequality stated in Proposition \ref{fund_ine_propt0} (since assumption (A4) holds), in the third inequality 
we used the fact that $v^\lambda$ is the minimizer of $\mathbb G^\lambda_0$,
in the fourth inequality we used the dominate convergence Theorem thanks to assumption (A4), in the fifth equality we exploited the fact that $\lim_{\lambda \to 0^+} \lambda^{q_0-p_0}=0$ for $q_0>p_0$,  {and in the sixth inequality we used \eqref{lavrentiev2}}. 
The symbol $I_{\delta,\tau}$ refers to the terms in round brackets.

Since \eqref{chainii} holds for any $0<\tau<\delta$, by first letting $\tau\to 0^+$ and then $\delta\to 0^+$, we have
\[
\lim_{\delta\to 0^+}\lim_{\tau\to 0^+}I_{\delta,\tau}=2\lim_{\delta\to 0^+}\int_{A_{\delta}\setminus A} |\nabla {\omega}(x)|^{p_0}dx= 0,
\] 
where in the first equality we exploited the uniform convergence of the Sobolev extension \cite[Prop. IV.21]{brezis1986analisi} and in the second equality we exploited the fact that the measure of $A_\delta\setminus A$ is negligible for $\delta \to 0^+$.

Hence,  {by the arbitrariness of $\varepsilon>0$,} the inequality \eqref{chainii} implies that $\mathbb B_0(v_B^0)= \mathbb B_0(v^0)$ and, therefore, $v_B^0=v^0$ thanks to the uniqueness of \eqref{Hii}. This result together with \eqref{conv_vt_v_0} yields the conclusion.
\end{proof}

\begin{rem}
We observe that \eqref{chainii}, implies the equality in the fundamental inequality \eqref{fundamental_inequality3_0}.

Let us observe that from \eqref{chainii}, we find that the fundamental inequality (i.e. the second inequality in \eqref{chainii}), also stated in Proposition \ref{fund_ine_propt0}, holds as equality:\[\mathbb B_0(v^0)= \liminf_{\lambda\to 0^+}\mathbb G_0^\lambda(v^\lambda).\]
\end{rem}

\section{The pointwise convergence assumption in the limiting cases}
\label{counter_sec}
The main aim of this Section is to prove that assumption (A4) for small Dirichlet data is sharp. Specifically, we provide one example where (A4) does not hold, and the previous convergence results (Theorems  \ref{Thm_conv_lim} and \ref{Thm_conv_limii}) do not hold.

 {With a similar approach, not reported here for the sake of brevity, it is possible to prove that even assumption (A4') is sharp.}

We prove this result by providing a Dirichlet energy density for which the ratio $Q_B(x,E)/E^{p_0}$ does not admit the limit for $E\to 0^+$. As before, we first need to provide two suitable subsequences (see Figure \ref{fig_9_counter} for the geometric interpretation).
\begin{figure}[ht]
    \centering
    \includegraphics[width=0.475\textwidth]{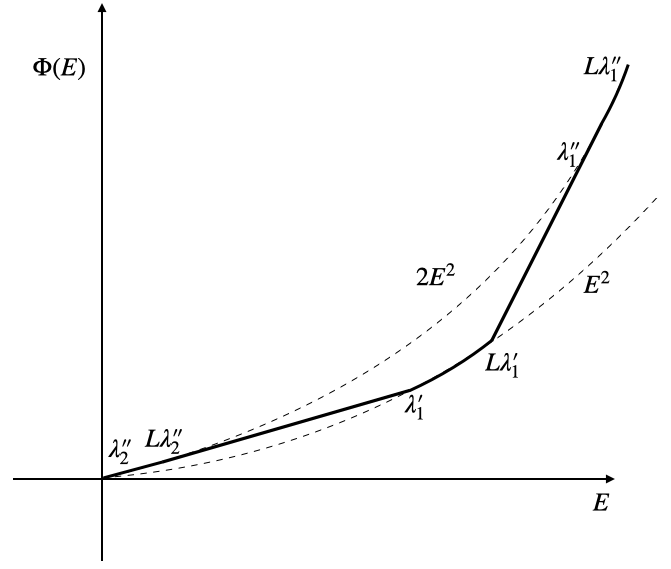}
    \includegraphics[width=0.475\textwidth]{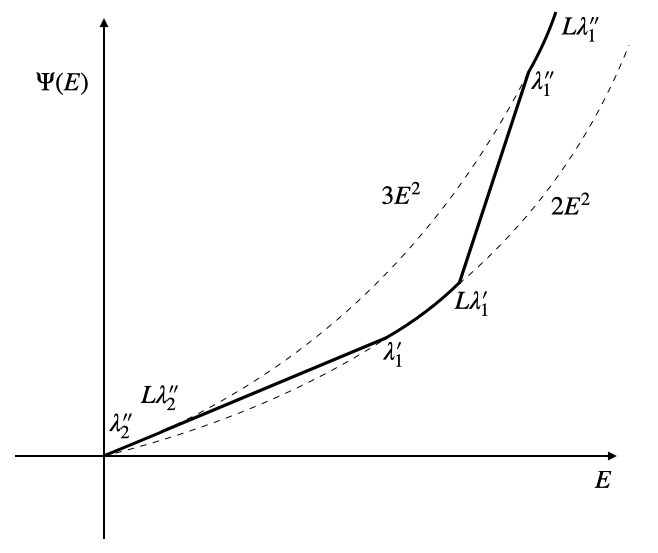}
    \caption{The continuous line represents the function describing the Dirichlet energy density used in the counterexample for small Dirichlet data.}
    \label{fig_9_counter}
\end{figure}

\begin{lem}\label{succ_L0}
Let $L>1$, then there exist two sequences
\[
\{\lambda_n'\}_{n\in\N}\downarrow 0^+ \quad\text{and}\quad \{\lambda_n''\}_{n\in\N}\downarrow 0^+
\]
such that
\[
L \lambda_{n+1}''<\lambda_n',\ 
L \lambda_n'< \lambda_n''\ \ 
\forall n\in\N,
\]
and a strictly convex function
\[
\Psi:[0,+\infty[\to[0,+\infty[
\]
 such that
\[
\Psi|_{[\lambda_n',L \lambda_n']}(E)=2E^2\quad \Psi|_{[\lambda_n'',L \lambda_n'']}(E)=3E^2.
\]
\end{lem}
\begin{proof}
Let us fix $\lambda_1''>0$.
For each $n \in \mathbb{N}$ we set the auxiliary function $\Phi$ equal to $2 E^2$ in $(\lambda_n'', L \lambda_n'')$ and equal to $E^2$ in $(\lambda_n', L \lambda_n')$. In interval $(L \lambda_n'',\lambda_n')$ the function $\Phi$ is equal to the tangent line to function $2E^2$ evaluated at $L \lambda_n''$. 
Point $\lambda_n'$ is found at the intersection of this tangent line with function $E^2$. In interval $(L \lambda_n',\lambda_{n+1}'')$ the function $\Phi$ is a straight line, continuous at $L \lambda_n'$ and tangent to $2E^2$.
Point $\lambda_{n+1}''$ is found as the abscissa of the tangent point between this straight line and function $2E^2$. This procedure is applied iteratively from $n=1$. 
Function $\Phi$ is convex and sequences $\{\lambda_n'\}_{n \in \mathbb{N}}$ and $\{\lambda_n''\}_{n \in \mathbb{N}}$ are monotonically decreasing to zero.

Therefore, the measure of intervals where $\Phi$ is equal to $E^2$ or equal to $2E^2$ is nonvanishing.
It is possible to prove that $\{\lambda'_n\}_{n\in\N}$ and $\{\lambda_n''\}_{n\in\N}$ are two  {geometric sequences}. Indeed
\[
\lambda_n'=\frac{1}{c_2L}\lambda_n'',\ \lambda_{n+1}''=\frac{1}{c_1L} \lambda_n',\ \lambda_{n+1}'=\frac 1{C^2L^2}\lambda_n',\ \lambda_{n+1}''=\frac 1{C^2L^2}\lambda_n'',
\]
where $c_1=2+\sqrt 2$, $c_2=1+\frac{\sqrt 2}{2}$ and $C$ is the  {geometric mean} of $c_1$ and $c_2$, that is $C^2
=\left(3+2\sqrt 2\right)$.

Finally, we set $\Psi(E)=\Phi(E)+E^2$. $\Psi$ is a strictly convex function.
\end{proof}

The construction of the counterexample in the planar case ($n=2$) for $p_0=2$ and $1<q_0<+\infty$ follows the steps of the previous Section line by line but with the aim of showing that $\lim_{\lambda\to 0^+}\mathbb G_0^\lambda(v^\lambda)\not\in\R$, where $\mathbb G_0^\lambda$ is defined in \eqref{G_norm}. Specifically, the two sequences $\{\lambda_n'\}_{n\in\N}\downarrow 0$ and $\{\lambda_n''\}_{n\in\N}\downarrow 0$ satisfy
\begin{equation*}
\limsup_{n\to +\infty}\mathbb G_0^{\lambda_n'}(v^{\lambda_n'})\le m_1(r)<m_2(r)\le\liminf_{n\to +\infty}\mathbb G_0^{\lambda_n''}(v^{\lambda_n''}).
\end{equation*}
The Dirichlet energy density defined as $Q_B(x,E)=\Psi(E)$, satisfies all the assumptions except (A4). This energy density is the basis to build a counterexample proving that (A4) is sharp. Specifically, we consider a 2D case ($n=2$) and $p_0=2$ in the outer region. The growth exponent $q_0$ satisfies condition $1<q_0<\infty$.

Let $r$ be greater than or equal to 10, and let the outer region $\Omega$ be the annulus centred in the origin with radii $1$ and $r$. This annulus is $D_r\setminus\overline D_1$,   {where $D_r$ and $D_1$ are the disks of radii $r$ and $1$, respectively,} and centered at the origin. The inner region is, therefore, $D_1$. We focus on problem \eqref{G_norm}, where the Dirichlet energy density is defined as
\[
\begin{split}
Q_B(x,E)&=\Psi(E)\quad\text{in}\ D_r\setminus\overline D_1\times [0,+\infty[,\qquad\\ Q_A(x,E)&=E^{q_0}\qquad\text{in}\ D_1\times [0,+\infty[.
\end{split}
\]
Let $\gamma$ be defined as $\gamma=7+\frac {12}{r^2}$. We denote $x=(x_1,x_2) \in \mathbb{R}^2$ and we consider the problem
\begin{equation}
\label{G^tc1large}
\min_{\substack{v\in W^{1,q}(D_r)\\ v=\gamma x_1\ \text{on}\ \partial D_r}}\mathbb G^\lambda(v),\quad \mathbb G^\lambda(v)=\frac 1 {\lambda^2}\left(\int_{ D_r\setminus D_1} \Psi(\lambda|\nabla v(x)|)dx+\int_{D_1} \lambda^q|\nabla v(x)|^q dx\right).
\end{equation}

Here we prove that $\lim_{\lambda\to +\infty}\mathbb G^\lambda(v^\lambda)$ does not exist. Specifically, the two sequences $\{\lambda_n'\}_{n\in\N}\uparrow +\infty$ and $\{\lambda_n''\}_{n\in\N}\uparrow+\infty$ of Lemma \ref{succ_L0} give 
\begin{equation}
   \label{counter_resultlarge}
\limsup_{n\to +\infty}\mathbb{G}^{\lambda_n'}(v^{\lambda_n'})\le\ell_1<\ell_2\le\liminf_{n\to +\infty}\mathbb{G}^{\lambda_n''}(v^{\lambda_n''}).
\end{equation}
As usual, $v^\lambda$ is the solution of \eqref{G^tc1large}.

Let us consider the following problem
\begin{align}
\label{problem_down}
&\min_{\substack{v\in H^{1}(D_r\setminus\overline D_1)\\ v=\gamma x_1\ \text{on}\ \partial D_r\\v=const.\ \text{on}\ \partial D_1}}\mathbb B_0 (v), \quad \mathbb B_0(v)=\int_{ D_r\setminus D_1} |\nabla v(x)|^2dx.
\end{align}
The symmetry of the domain and the zero average of the boundary data imply that the constant appearing in \eqref{problem_down} on $\partial D_1$ is zero.

An easy computation reveals that
\[
v_{D_r}(x)=\frac{7r^2+12}{r^2-1}\left(1-\frac{1}{x_1^2+x_2^2}\right)x_1\quad\text{in}\ D_r\setminus\overline D_1
\]
is the solution of 
\eqref{problem_down}, that $\Delta v_{D_r}=0$ in $D_r\setminus\overline D_1$, and that we have
\[
\begin{split}
\frac{7r^2-12}{r^2-1}\left(1-\frac{1}{\rho^2}\right)\le|\nabla v_{D_r}(x)|\le\frac{7r^2+12}{r^2-1} \left(1+\frac 1{\rho^2}\right)\quad\text{on}\ \partial D_{\rho},\ 1<\rho\le r.
\end{split}
\]
Consequently, when $\rho\ge2$, we have
\begin{equation}
    \label{stima_nablaw}
1\le\frac 34 \frac{7r^2-12}{r^2-1}\le|\nabla v_{D_r}(x)|\le\frac 54 \frac{7r^2+12}{r^2-1}\leq 10\quad\text{in}\ D_r\setminus D_2.
\end{equation}

Let $L$ be greater than 10, $\lambda_n'\uparrow +\infty$ and let $\lambda_n''\uparrow +\infty$ be the  two sequences of Lemma \ref{succ_L0}.  
It turns out that
\begin{equation}
\label{stime_up_down}
\lambda_n' \le
\lambda_n'|\nabla v_{D_r}(x)|
\le L \lambda'_n\quad\text{in}\ D_r\setminus D_2.
\end{equation}
We have
\begin{equation}
    \label{inf_counter}
\begin{split}
&\limsup_{n\to +\infty}\mathbb G^{\lambda'_n}(v^{\lambda_n'})\leq\limsup_{n\to +\infty} \mathbb G^{\lambda'_n}(v_{D_r})\\
&=\limsup_{n\to +\infty}\frac 1{(\lambda'_n)^2} \int_{D_r\setminus D_2}\Psi(\lambda'_n|\nabla v_{D_r}(x)|)dx+\limsup_{n\to +\infty}\frac 1{(\lambda'_n)^2} \int_{D_2\setminus D_1}\Psi(\lambda'_n|\nabla v_{D_r}(x)|)dx\\
&\leq 2 \int_{D_r\setminus D_2}|\nabla v_{D_r}(x)|^2dx+3 \int_{D_2\setminus D_1}|\nabla v_{D_r}(x)|^2dx,
\end{split}
\end{equation}
where in the first line we used the minimality of $v^{\lambda'_n}$ for $\mathbb G^{\lambda'_n}$ and that $v_{D_r}$ is an admissible function for problem \eqref{G^tc1large},
in the second line we exploited the property that the gradient of $v_{D_r}$ in $D_1$ is vanishing, and in the third line we used \eqref{stime_up_down}. By setting $\ell_1$ equal to \eqref{inf_counter}:
\[
\ell_1:=2 \int_{D_r\setminus D_2}|\nabla v_{D_r}(x)|^2dx+3 \int_{D_2\setminus D_1}|\nabla v_{D_r}(x)|^2dx,\]
we have the leftmost inequality in \eqref{counter_resultlarge}.

To obtain the rightmost inequality in \eqref{counter_resultlarge}, we consider the following problems
\begin{equation}
    \label{AuxF}
\min_{\substack{v\in H^{1}(D_r\setminus\overline D_1)\\ v=\gamma x_1\ \text{on}\ \partial D_r}} \mathbb H^\lambda (v),\quad\mathbb H^\lambda(v)= \frac{1}{\lambda^2}\int_{D_r\setminus D_2}\Psi(\lambda|\nabla v(x)|)dx+2 \int_{D_2\setminus D_1}|\nabla v(x)|^2dx.
\end{equation}
\begin{equation}
    \label{problem_up}
\min_{\substack{v\in H^{1}(D_r\setminus\overline D_1)\\ v=\gamma x_1\ \text{on}\ \partial D_r}}\mathbb D(v),\quad \mathbb D (v) =3\int_{D_r\setminus D_2}|\nabla v(x)|^2dx+2\int_{D_2\setminus D_1}|\nabla v(x)|^2dx.
\end{equation}
The unique solution of \eqref{problem_up} is
\[
w_{D_r}(x)=
\begin{cases}
\left(7+\frac{12}{x_1^2+x_2^2}\right)x_1&\quad\text{in}\ D_r\setminus\overline D_2\\
8\left(1+\frac{1}{x_1^2+x_2^2}\right)x_1 &\quad\text{in}\ D_2\setminus\overline D_1.
\end{cases}
\]
Analogously to \eqref{stima_nablaw}, it can be easily proved that
\[
1\le 4\le \left(7-\frac{12}{\rho^2}\right)\le|\nabla w_{D_r}(x)|\le \left(7+\frac {12}{\rho^2}\right)\le 10< L\quad\text{on}\ \partial D_{\rho},\ 2\le\rho\le r.
\]
and hence we choose $L>10$ such that
\begin{equation}
\label{stime_up}
\lambda_n'' \le
\lambda_n''|\nabla w_{D_r}(x)|
\le L \lambda''_n\quad\text{in}\ D_r\setminus D_2.
\end{equation}
Therefore, we have
\begin{equation}
\label{sup_counter}
\mathbb{G}^{\lambda_n''}(v^{\lambda_n''})\geq\mathbb H^{\lambda''_n}(w_{D_r})=\mathbb D(w_{D_r}),
\end{equation}
where the inequality comes from the definition of $\Psi$. The equality follows from the fact that $\mathbb H^{\lambda''_n}$ coincides with $\mathbb D$ by \eqref{stime_up} and the definition of $\Psi$. Note that $w_{D_r}$ is a local minimizer in $W^{1,\infty}(D_r)\cap W^{1,2}(D_r)$, since $\Psi$ does not depend on $x$  (see \cite{cianchi2010global} for details). Finally, $w_{D_r}$ is a global minimizer thanks to the uniqueness of \eqref{AuxF}.

By setting
\[\ell_2(r):=
\mathbb D(w_{D_r}).
\]
we have the rightmost inequality in \eqref{counter_resultlarge} by passing to the limit in \eqref{sup_counter}.

At this stage, it only remains to be proved that $\ell_1(r)<\ell_2(r)$. To this purpose, we notice that:
\[
\begin{split}
 \ell_1(r)&= 2 \int_{D_r\setminus D_2}|\nabla v_{D_r}(x)|^2dx+3 \int_{D_2\setminus D_1}|\nabla v_{D_r}(x)|^2dx\\
\ell_2(r)&= 3\int_{D_r\setminus D_2}|\nabla w_{D_r}(x)|^2dx+2\int_{D_2\setminus D_1}|\nabla w_{D_r}(x)|^2dx.
\end{split}
\]
Condition $\ell_1(r)<\ell_2(r)$ holds for large $r$, by observing that (i) $v_{D_r}$ and $w_{D_r}$ solve the same associated Euler-Lagrange equation on $D_2\setminus\overline D_1$, (ii) $\nabla v_{D_r}(x)$ and $\nabla w_{D_r}(x)$ are bounded functions on the bounded domain $D_2\setminus D_1$ by \eqref{stime_up_down} and \eqref{stime_up}, respectively, and (iii) it turn out that
\[
\begin{split}
&\lim_{r\to+\infty}\frac{
\int_{ D_r\setminus D_2} |\nabla v_{D_r}(x)|^2dx}{\int_{ D_r\setminus D_2} |\nabla w_{D_r}(x)|^2dx}= 1,\\
&\lim_{r\to+\infty}\int_{D_r\setminus D_2}|\nabla v_{D_r}(x)|^2dx=\lim_{r\to+\infty}\int_{D_r\setminus D_2}|\nabla w_{D_r}(x)|^2dx=+\infty.
\end{split}
\]

\section{Forward and Inverse Problems: Applications and Numerical analysis}
\label{num_sec}

\begin{figure}[htp]
    \centering
    \includegraphics[width=1.0\textwidth]{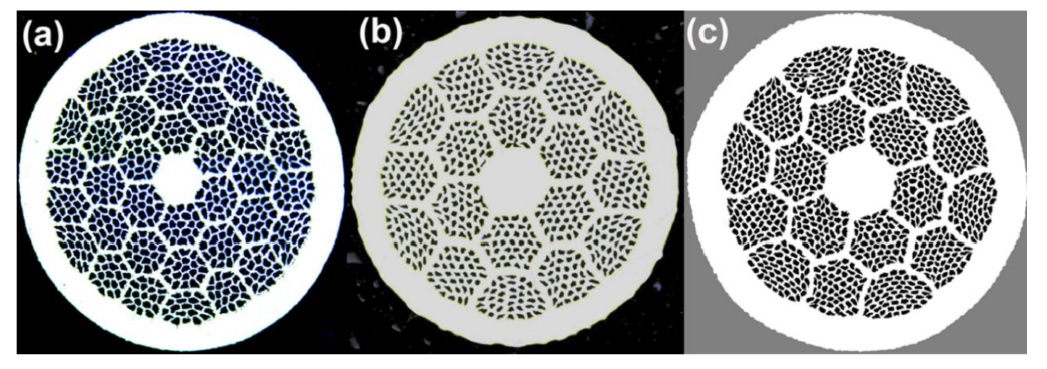}
    \caption{Picture of the cross section for typical superconducting cables. The cable consists in several petals (36 petals for (a), and 18 petals for (b) and (c)). Each petal is made up of many thin SC wires (19 wires for (a), 37 wires for (b) and 61 for (c)). The picture is in \cite[Fig. 4]{instruments4020017} and it is courtesy of Instruments-MPDI.}
    \label{fig_9_SC}
\end{figure}
In this Section we propose some applications of the theoretical results of the previous Sections. The case of study refers to superconducting wires: a major component in technological applications. After a brief presentation of superconducting materials, we show the impact of the theoretical results on both the \emph{Forward Problem}, i.e. finding the scalar potential $u$ assigned to the materials and the boundary data, and the \emph{Inverse Problem}, i.e retrieving the shape of defects in the cross section of the wire. Figure \ref{fig_9_SC} shows a typical cross section for a few superconducting cables.

A type II High Temperature Superconducting (HTS) material \cite{seidel2015applied,krabbes2006high}, in its superconductive state, is well described by a constitutive relationship given by 
\begin{equation}
    \label{E-J_power_law}
E(J)=E_0\left({J}/{J_c}\right)^n.
\end{equation}
This constitutive relationship, named \textit{E-J Power Law}, was proposed by Rhyner in \cite{rhyner1993magnetic} to properly reflect the nonlinear relationship between the electric field and the current density in HTS materials.

An HTS described by an ideal E-J Power Law behaves like a PEC for a \lq\lq small\rq\rq \ boundary potential, when \lq\lq immersed\rq\rq\ in a linear conductive material. Indeed, its electrical conductivity is given by
\begin{equation*}
  \sigma = \frac {J_c}{E_0} \left( \frac{E}{E_0} \right) ^\frac{1-n}n,
\end{equation*}
and it turns out that $\sigma\to+\infty$ as $E \to 0^+$.

Typical parameters for $J_c$ and $n$ are given in Table \ref{table_par}. They refer to commercial products from European Superconductors (EAS-EHTS) and  American Superconductors (AMSC) \cite{lamas2011electrical}. The value of $E_0$ is almost independent of the material and equal to $0.1$ mV/m \cite{yao2019numerical}.

\begin{table}[htp]
\centering
\hspace{2.4cm}\begin{tabular}{|l|l|l|}
\hline 
Type & $J_c$[A/mm$^2$] & $n$\\ \hline
BSCCO EAS  & 85 & 17  \\ 
BSCCO AMSC & 135 & 16 \\
YBCO AMSC & 136 & 28 \\
YBCO SP SF12100 & 290 & 30 \\
YBCO SP SCS12050 & 210 & 36 \\ \hline
\end{tabular}
\caption{Typical parameters for $J_c$ and $n$.}
\label{table_par}
\end{table}

The numerical examples have been developed with an in-house Finite Element Method (FEM) based on \cite{grilli2005finite}.
We consider a standard Bi-2212 round wire. The geometry of the cable is shown in Figure \ref{fig_10_mesh} (left), together with the finite element mesh used for the numerical computations (see Figure \ref{fig_10_mesh} (right)). 

\begin{figure}[htp]
    \centering
    \includegraphics[width=0.45\textwidth]{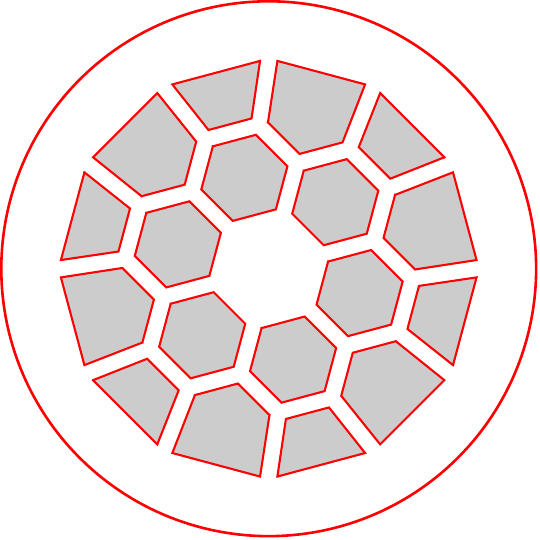}\qquad\quad \includegraphics[width=0.45\textwidth]{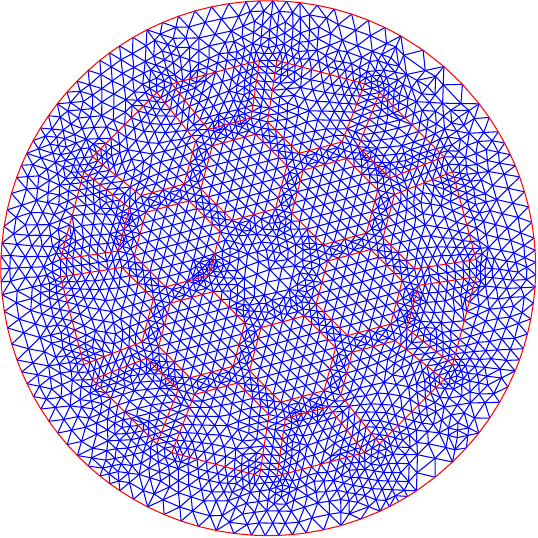}
    \caption{Left: geometry of the cross-section of the superconducting cable. The solid superconducting material (light grey) in the linear matrix (white). Right: the finite element mesh used in the numerical computations.}
    \label{fig_10_mesh}
\end{figure}

The radius $R_e$ of this HTS cable is equal to 0.6 mm \cite{huang2013bi}. The geometry of the problem is simplified w.r.t. those in Figure \ref{fig_9_SC}. Specifically, each \lq\lq petal\rq\rq\ is assumed to be made up of an individual (solid) superconducting wire, rather than many thin superconducting wires. The electrical conductivity of the matrix surrounding the petals is $5.55 \times 10^7$ S/m \cite{li2015rrr}. This electrical conductivity is equal to $95.8\%$ IACS. As reported in~\cite{Ba21}, the superconducting material is characterized by a very high value of critical current $J_c$. In particular, we assume
\begin{displaymath}
J_c=8000\,\text{A/mm$^2$}, \quad n=27.
\end{displaymath}

\subsection{Solution of the forward problem}\label{forward_problem}
The replacement of the original problem with its limiting \lq\lq version\rq\rq\ has a major impact when numerically solving the \emph{Forward Problem}, i.e. the computation of the scalar potential $u$ for prescribed materials and boundary data. Specifically, the solution of the original nonlinear problem is carried out by an iterative method. At each iteration the solution of a linear system of equations is required. This linear system of equations is characterized by a sparse matrix and, therefore, its solution is obtained by an iterative approach. However, the regions corresponding to the nonlinear materials may give rise to strongly ill-conditioned matrices, posing relevant challenges when solving the related linear system of equations. On the other hand, when it is possible to replace the original problem with its limiting version, the nonlinear material is replaced by a PEC and the overall problem is linear.

In the following we compare the solution of the \emph{Forward Problem} obtained by simulating the actual (nonlinear) superconducting material and the limiting problem where the superconducting material is replaced by the PEC. 

In all numerical calculations, the applied boundary potential is equal to $f(x,y)=V_0 x/R_e$. In this way the parameter $V_0$ represents the maximum value for the boundary potential.

Figure \ref{fig_11_err} shows the errors for \lq\lq small\rq\rq\ boundary potential when replacing the superconducting material with a PEC. The error metrics are equal to the relative error in the $L^2-$norm ($e_2$) and in the $L^\infty-$norm ($e_\infty$). As expected, the PEC approximation is valid for \lq\lq small\rq\rq\ $V_0$. The accuracy of the approximation is very high in its domain of validity.

\begin{figure}[htp]
    \centering
    \includegraphics[width=0.7\textwidth]{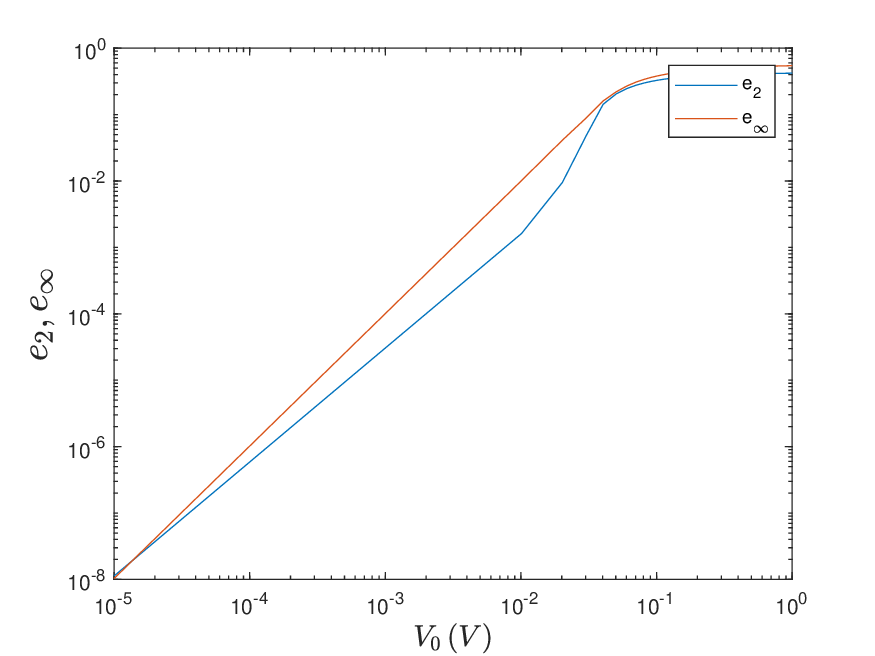}
    \caption{Plot of the error for the PEC approximation (\lq\lq small\rq\rq\ boundary potential).}
    \label{fig_11_err}
\end{figure}

Figure \ref{fig_12_E_CEP_J_IEP} shows the spatial distribution of the electric field $\mathbf{E}$ (top) and of the electric scalar potential $u$ (bottom) for small ($V_0=10^{-6}$V) boundary potential, in the presence of the actual HTS cable. It is evident from the plot that $\mathbf{E}$ is perpendicular to the HTS regions for small $V_0$. This is in line with the concept that for small $V_0$ the HTS regions behave like a PEC, where $\mathbf{E}$ is orthogonal to their boundaries.

\begin{figure}[htp]
    \centering
    \includegraphics[width=0.7\textwidth]{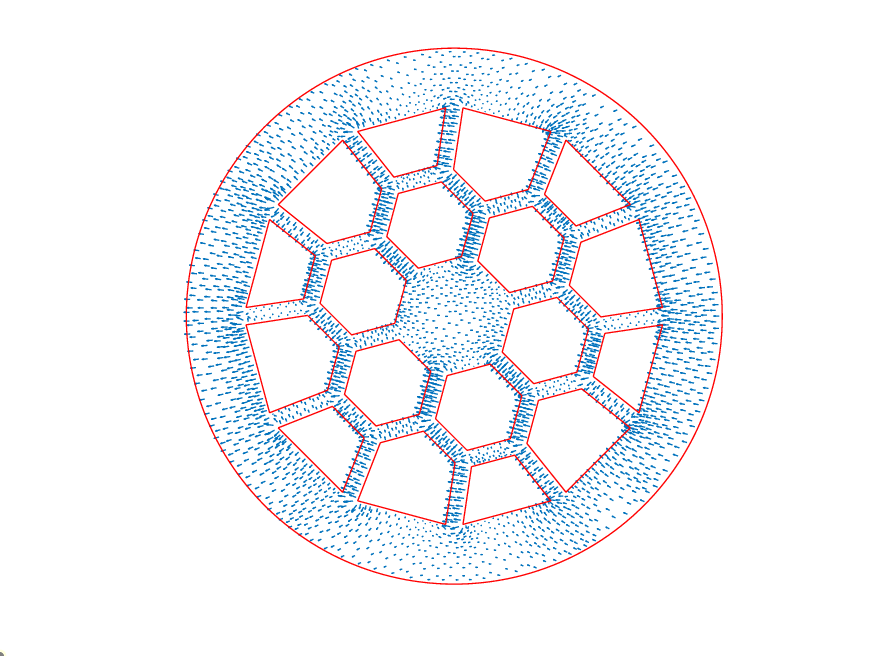}\quad\qquad
    \includegraphics[width=0.7\textwidth]{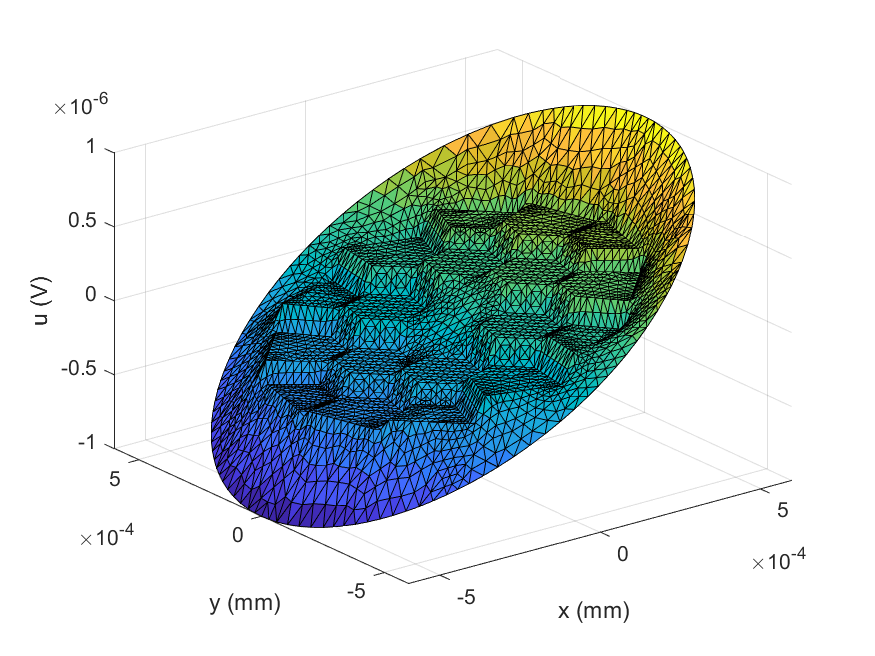}
    \caption{Plot of the electric field $\mathbf{E}$ (top) and of the electric scalar potential $u$ (bottom) for a \lq\lq small\rq\rq \ boundary potential ($V_0=10^{-6}$ V).}
    \label{fig_12_E_CEP_J_IEP}
\end{figure}

\subsection{Imaging via linear methods}
In this Section we provide numerical examples related to the solution of the inverse problem. We treat the case of $p_0=2$ since the limiting problem, being linear, provides a powerful \lq\lq bridge\rq\rq\ toward well assessed and mature methods and algorithms developed for linear problems.

From a more general perspective ($p_0 \ne 2$), the limiting case approach is very relevant because it brings to the light that, when facing an inverse problem with nonlinear materials, the canonical problem consists in solving an inverse problem for the  {weighted} $p_0-$Laplace equation, regardless of the specific nature of the nonlinearity.

The configuration is that of section \ref{forward_problem}, and it refers to a superconducting cable. Specifically, the inverse problem consists in retrieving the shape of defects in the Mg-Al alloy matrix of a superconducting wire, starting from boundary data. This is a very challenging task because of the nonlinear behaviour of the superconductive petals, which results in a nonlinear relationship between the applied boundary voltages and the measured boundary currents.

From a general perspective, the inverse problem can be tackled as follows: (i) the data are collected on the actual (nonlinear) configuration under small boundary data operations and (ii) the data are processed assuming that the governing equations are those for the limiting problem that, in this case, is linear.

As mentioned above, this means that the inverse problem of retrieving the position and shape of anomalies in the Mg-Al matrix can be addressed by means of the imaging methods developed for linear materials.

\subsubsection{The imaging algorithm}
The imaging algorithm attempts to estimate a tomographic reconstruction of the shape and position of the anomalies in the Mg-Al matrix, starting from boundary data. The boundary data we adopt is the Dirichlet-to-Neumann (DtN) operator, which maps a prescribed boundary potential into the corresponding normal component of the electrical current density entering through $\partial \Omega$. In other words, the DtN operator gives the voltage-to-current relationship on the boundary $\partial \Omega$.

In a practical setting we can measure a noisy version of a discrete approximation of the DtN operator, that is a noisy version of the DtN operator restricted to a finite dimensional linear subspace of applied voltages, as in the case of a finite set of boundary electrodes. Hereafter, this discrete approximation is referred as $\mathbf{G}$, the conductivity matrix.

In order to solve the inverse problem we adopt an imaging method based on the  Monotonicity Principle~\cite{Tamburrino_2002}, here briefly summarized for the sake of completeness.

Both the DtN operator and its discrete counterpart, the conductances matrix $\mathbf{G}$, satisfy a Monotonicity Principle (see \cite{gisser1990electric, Tamburrino_2002,corboesposito2021monotonicity}):
\begin{equation}\label{eqn:mono1}
    \sigma_1\leq\sigma_2 \Longrightarrow \mathbf{G}_1 \preceq \mathbf{G}_2,
\end{equation}
where $\sigma_1 \le \sigma_2$ is meant a.e. in $\Omega$ and $\mathbf{G}_1 \preceq \mathbf{G}_2$ means that $\mathbf{G}_1 - \mathbf{G}_2$ is a negative semi-definite matrix.
The Monotonocity Principle states a monotonic relationship between the pointwise values of the electrical conductivities and the measured boundary operator.

The targeted problem, that is the imaging of anomalies in the Mg-Al phase of a superconducting cable, is a two phase problem. One phase corresponds to the healthy material (the Mg-Al phase plus the PEC which replaces the superconductor, at small boundary data) while the other phase corresponds to the damaged region, having an electrical conductivity $\sigma_I$ smaller than the electrical conductivity $\sigma_{BG}$ of the Mg-Al phase. As a consequence, if $V_1$ and $V_2$ are two possible anomalies well-contained in the Mg-Al phase, and $V_1 \subseteq V_2$, it turns out that $\sigma_1 \ge \sigma_2$ and, therefore
\begin{equation}\label{eqn:mono2}
    V_1 \subseteq V_2 \Longrightarrow \mathbf{G}_{V_1} \succeq \mathbf{G}_{V_2}
\end{equation}

Equation~\eqref{eqn:mono2} can be translated in terms of an imaging method, as originally proposed in \cite{Tamburrino_2002}. Specifically, \eqref{eqn:mono2} is equivalent to $\mathbf{G}_{V_1} \nsucceq \mathbf{G}_{V_2} \Longrightarrow V_1 \nsubseteq V_2$, which gives for $V_1=T$ and $V_2=V$
\begin{equation}
    \label{eqn:mono3}
    \mathbf{G}_{T} \nsucceq \mathbf{G}_{V} \Longrightarrow T \nsubseteq V,
\end{equation}
where $V$ is the domain occupied by the unknown anomaly and $T$ a known domain occupied by test domain. Equation \eqref{eqn:mono3} makes it possible to infer whether the test domain $T$ is not contained in the unknown region occupied by the anomaly $V$. By repeating this type of test for several prescribed test anomalies occupying regions $T_1, \ T_2, \ \ldots$, we get an estimate ${V}_U$ of the unknown anomaly $V$ as:
\begin{align*}
    V_U=\bigcup_{k\in\Theta} T_k, \qquad
    \Theta=\{T_k\,|\,\mathbf{G}_{V}-\mathbf{G}_{T_k}\succeq 0\}.
\end{align*}
In the presence of noise, as is the case in any real-world problem, we adopt a slightly different strategy (see \cite{harrach2015resolution,Tamburrino2016284}). Specifically, let $\tilde{\mathbf{G}}_V=\mathbf{G}_V+\mathbf{N}$
be the noisy version of the data, where $\delta$ is an upper bound to the 2-norm of the noise matrix $\mathbf{N}$, i.e. $\norm{\mathbf{N}}_2\leq\delta$, then the reconstruction is obtained as
\begin{equation}\label{eq:rule}
    \tilde{V}_U =\bigcup_{k\in\Theta} T_k, \qquad \Theta=\{T_k\,|\,\tilde{\mathbf{G}}_{V}-\mathbf{G}_{T_k}+\delta\,\mathbf{I}\succeq 0\}
\end{equation}
where $\mathbf{I}$ is the identity matrix. Moreover, it has been shown that, under proper assumptions $V \subseteq \tilde{V}_U$ even in the presence of noisy data \cite{Tamburrino2016284}. With a similar imaging method it is possible to reconstruct a lower bound $\tilde{V}_L$ to $V$. Existence of upper ($\tilde{V}_U$) and lower ($\tilde{V}_L$) bounds to the unknown anomaly is an unique feature of this imaging method.

Finally, the matrices related to test domains $T_k$ are evaluated numerically by replacing the superconducting petals with perfect electric conductors.

\subsubsection{Numerical Results}

\begin{figure}[htp]
    \centering
    \includegraphics[width=0.4\textwidth]{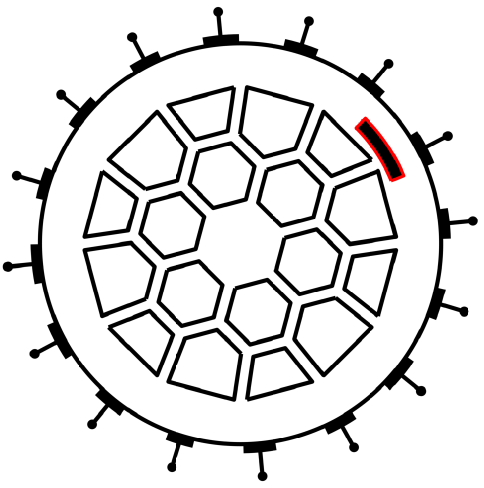}\quad\qquad
    \includegraphics[width=0.4\textwidth]{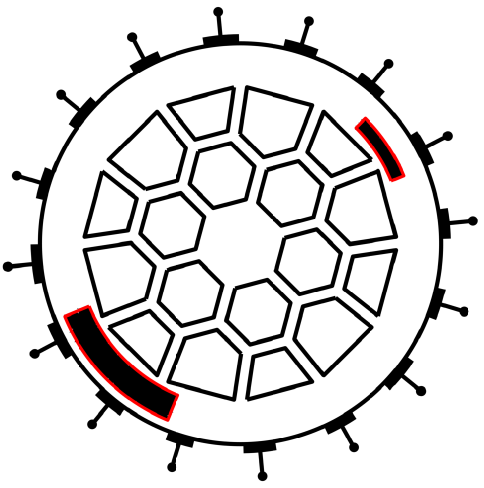} \\
    \includegraphics[width=0.4\textwidth]{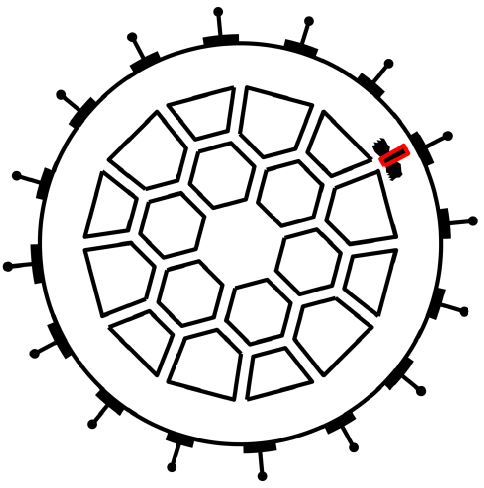}\quad\qquad
    \includegraphics[width=0.4\textwidth]{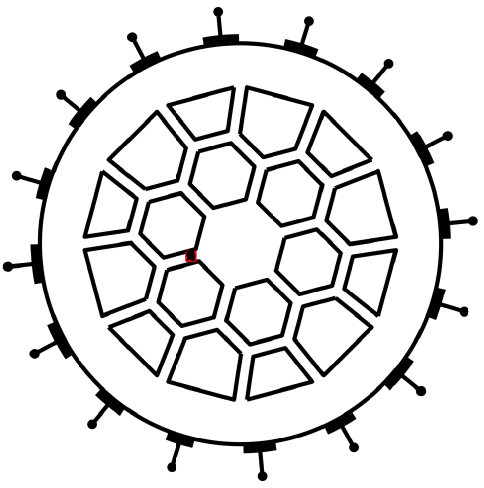} \\
    \includegraphics[width=0.4\textwidth]{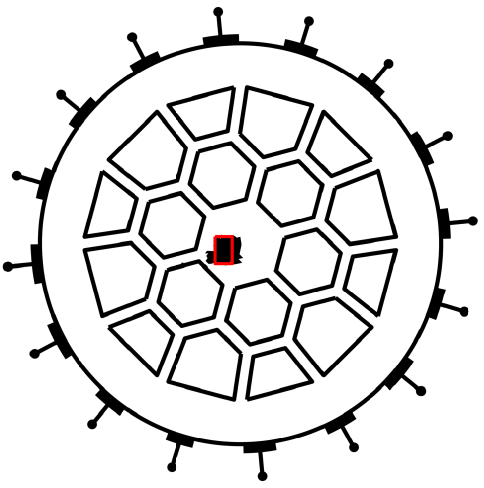}\quad\qquad
    \includegraphics[width=0.4\textwidth]{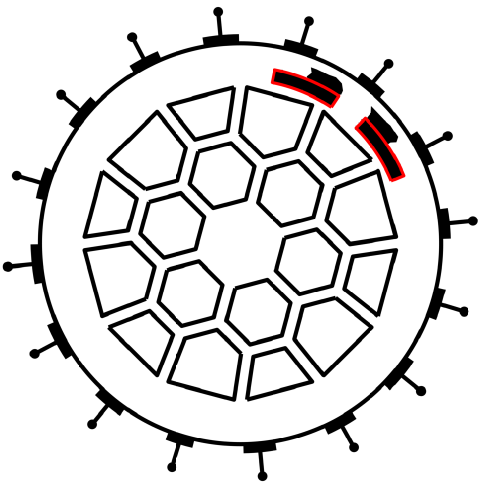}
    \caption{Reconstructions obtained by means of monotonicity based algorithm for linear materials in the \lq\lq small\rq\rq\ boundary data regime.}
    \label{fig_13_CEP}
\end{figure}

The reference geometry is that of the superconducting cable in Figure~\ref{fig_10_mesh} (left). We apply 16 electrodes at the boundary. The electrodes are equal and uniformly spaced.

We assume the following noise model
\begin{displaymath}
    \mathbf{N}=\eta \Delta G_{max} \mathbf{A}
\end{displaymath}
where $\mathbf{A}$ is a random matrix belonging to the Gaussian Orthogonal Ensamble (GOE) (see \cite{mehta2004random} for details),
\begin{displaymath}
    \Delta G_{max}=\max_{i,j}\abs{(G_V)_{ij}-(G_{BG})_{ij}},
\end{displaymath}
$\mathbf{G}_{BG}$ is the conductance matrix associated to the healthy superconducting wire, i.e. without defects, and $\eta$ is a parameter representing the (relative) noise level.

Hereafter we assume the noise level $\eta$ equal to $0.01$. Figure~\ref{fig_13_CEP} shows the reconstructions obtained with this method. The reconstructed defects are represented in black whereas the red line is the boundary of the real defects. The boundary of the petals is reported in black, like the electrodes but with a thicker line. These reconstructions clearly demonstrate the effectiveness of the approach.

The conductance matrix was evaluated by applying boundary voltages of $1$mV.

\section{Conclusions}
\label{Con_sec}
This study is a contribution to Inverse Problems in the presence of nonlinear materials. This subject is still at an early stage of development, as stated in \cite{lam2020consistency}. 

We focus on Electrical Resistance Tomography where the aim is to retrieve the electrical conductivity/resistivity of a material by means of stationary (DC) currents. Our main results consist prove that the original nonlinear problem can be replaced by a proper  {weighted} $p_0-$Laplace problem, when the prescribed Dirichlet data is small\rq\rq. Specifically, we prove that in the presence of two different materials, where at least one is nonlinear, the scalar potential in the outer region in contact with the boundary where the Dirichlet data is prescribed, can be computed by (i) replacing the interior region with either a Perfect Electric Conductor ($q_0<p_0$) or a Perfect Electric Insulator ($q_0<p_0$) and (ii) replacing the original problem (material) in the outer region with a  {weighted} $p_0-$Laplace problem. The presence of the \lq\lq fingerprint\rq\rq\ of a  {weighted} $p_0-$Laplace problem can be recognized to some extent in an arbitrary nonlinear problem. From the perspective of tomography, this is a significant result because it highlights the central role played by the $p_0-$Laplacian in inverse problems with nonlinear materials. For $p_0=2$, i.e. when the material in the outer region is linear, these results constitute a powerful bridge allowing all theoretical results, imaging methods and algorithms developed for linear materials to be brought into the arena of problems with nonlinear materials. 

The fundamental tool to prove the convergence results is the inequality appearing in Proposition \ref{fund_ine_propt0}. These results express the asymptotic behaviour of the Dirichlet energy for the outer region in terms of a factorized $p_0-$Laplacian form. Moreover, we have proved that our assumptions are sharp, by means of proper counterexamples.

Finally, we provide a numerical example, referring to a superconducting cable, as an application of the theoretical results proved in this paper.


\section*{Acknowledgements}
This work has been partially supported by the MiUR-Progetto Dipartimenti di eccellenza 2018-2022 grant \lq\lq Sistemi distribuiti intelligenti\rq\rq of Dipartimento di Ingegneria Elettrica e dell'Informazione \lq\lq M. Scarano\rq\rq, by the MiSE-FSC 2014-2020 grant \lq\lq SUMMa: Smart Urban Mobility Management\rq\rq\ and by GNAMPA of INdAM.

\section*{Data availability Statements}
The datasets generated during and/or analysed during the current study are available from the corresponding author on reasonable request.

\bibliographystyle
{ieeetr}
\bibliography{biblioCFPPT}

\end{document}